\RequirePackage{fix-cm}
\documentclass[smallextended]{svjour3}       
\smartqed  
\usepackage{graphicx}
\usepackage{amsmath}
\usepackage{amssymb}
\usepackage[usenames]{color}
\usepackage{epsfig}
\newtheorem{assumption}{Assumption}
\journalname{}
\begin{document}

\title{Motion planning and stabilization of nonholonomic systems using gradient flow approximations}

\titlerunning{Motion planning and stabilization of nonholonomic systems}        

\author{Victoria Grushkovskaya \and Alexander Zuyev}


\institute{V. Grushkovskaya \at
             University of Klagenfurt, Universitätsstraße 65-67
9020 Klagenfurt am Wörthersee, Austria
\\ and
\\
Insitute of Applied Mathematics and Mechanics, National Academy of Sciences of Ukraine (IAMM NASU) \\
              \email{viktoriia.grushkovska@aau.at}           
           \and
           A. Zuyev (Corresponding author) \at
              Otto von Guericke University Magdeburg, Universitätsplatz~2, 39106 Magdeburg, Germany
              \\
              and\\
              IAMM NASU\\
              \email{zuyev@mpi-magdeburg.mpg.de}
}

\date{Received: \today}

\maketitle

\begin{abstract}
Nonlinear control-affine systems with time-varying vector fields are considered in the paper. We propose a unified control design scheme with oscillating inputs for solving the trajectory tracking and stabilization problems.
This methodology is based on the approximation of a gradient like dynamics by trajectories of the designed closed-loop system.
As an intermediate outcome, we characterize the asymptotic behavior of solutions of the considered class of nonlinear control systems with oscillating inputs under rather general assumptions on the generating potential function.
These results are applied to examples of nonholonomic trajectory tracking and obstacle avoidance.
\end{abstract}

\section{Introduction}
\label{sec_intro}

Consider a nonlinear control system
\begin{equation}
\dot x = \sum_{k=1}^m f_k (x,t)u_k, 
\label{Sigma}
\end{equation}
where $x=(x_1,...,x_n)^\top\in D\subseteq\mathbb R^n$ is the state, $u=(u_1,...,u_m)^\top\in\mathbb R^m$ is the control,
$D$ is a domain, and the time dependent vector fields $f_k:D\times {\mathbb R}^+\to {\mathbb R}^n$ are regular enough to guarantee the existence and uniqueness of solutions to the Cauchy problem for system~\eqref{Sigma} with any initial data $x(t_0)=x^0\in D$, $t_0\ge 0$, and any admissible control $u:[t_0,+\infty)\to {\mathbb R}^m$. We will formulate the required regularity assumptions precisely below.

The driftless control-affine system~\eqref{Sigma} is an extremely important mathematical model in nonholonomic mechanics,
which represents the kinematics with non-integrable constraints in the case $m<n$ (we refer to the book~\cite{Bloch15} for general reference).
Of special interest is the class of systems with time independent vector fields:
\begin{equation}
\dot x = \sum_{k=1}^m \tilde f_k (x)u_k,\quad x\in D\subseteq\mathbb R^n,\; u\in\mathbb R^m,\; \tilde f_k\in C^1(D;{\mathbb R}^n).
\label{Sigma_a}
\end{equation}
In contrast to linear control theory, the controllability of system~\eqref{Sigma_a} does not imply its stabilizability by a
regular feedback law of the form $u=h(x)$.
A~famous example of a completely controllable system~\eqref{Sigma_a} with $n=3$ and $m=2$, which is not stabilizable in the classical sense, was presented in~\cite{Bro83}.
Since then, the stabilization and motion planning problems of nonholonomic systems have been extensively studied by many experts in nonlinear control theory, mechanics, and robotics. A survey of essential contributions in this area is performed in Section~\ref{sec_related}.

To the best of our knowledge, the present paper contains the first description of a unified control design method for solving a variety of different control problems such as: stabilization of an equilibrium point $x=x^*$, tracking an arbitrary curve in the state space, and motion planning with obstacles for rather general non-autonomous systems~\eqref{Sigma}.
The main idea behind our construction 
is to design time dependent feedback controllers in such a way that the trajectories of the corresponding closed-loop system approximate the trajectories of a gradient-like system of the form
\begin{equation}\label{sys_grad}
  \dot {x}=-\gamma\frac{\partial}{\partial x}P( x,t),\quad  x\in\mathbb R^n,
\end{equation}
where the {potential} function $P(x,t)$ and gain $\gamma>0$ are to be defined according to the specific problem statement.
The key contribution of our work is twofold:
\begin{itemize}
  \item a unified approach for solving the stabilization and motion planning problems
  under the bracket-generating condition;
  \item relaxed regularity assumptions on the vector fields and their directional derivatives. In particular, vector fields of the considered class of systems are not required to be smooth.
\end{itemize}

The subsequent presentation is organized as follows.
The outcomes of the literature study are reported in Section~\ref{sec_related}.
A family of $\varepsilon$-periodic feedback controllers is introduced in Section~\ref{sec_unified} in the form of trigonometric polynomials with respect to time with coefficients depending on the system state.
It is shown in Subsection~\ref{sec_basic} that the proposed controllers allow approximating the reference gradient flow dynamics by the trajectories of system~\eqref{Sigma} with arbitrary accuracy under a suitable choice of the small parameter $\varepsilon$.
These approximation schemes are then adapted to derive stabilizing controllers for the equilibrium stabilization problem (Theorem~\ref{thm_step1stab} and its corollary in Subsection~\ref{sec_stab}), tracking problem (Theorem~\ref{thm_step1track} and its corollary in Subsection~\ref{sec_tracking}), and obstacle avoidance (Subsection~\ref{sec_obstacle}).
We illustrate the proposed control design methodology with examples in Section~\ref{sec_exmpls1}.
Finally, concluding comments are given in Section~\ref{sec_concl} to summarize the key results of the present paper and underline its contribution with respect to the previous work.
The proofs of the main results are contained in Appendices~A--D.

\subsection{Notations}

Throughout the text, we will use the following notations:

$\mathbb R^+$ -- the set of nonnegative real numbers;

$\mathbb R_{>0}$ -- the set of positive real numbers;


$\delta_{ij}$ -- the Kronecker delta: $\delta_{ii}{=}1$ and $\delta_{ij}{=}0$ whenever $i\ne j$;


${\rm dist}(x,S)$ -- the Euclidian distance between a point $x\in\mathbb R^{n}$ and a set

$B_\delta(x^*)$ -- $\delta$-neighborhood of an $x^*\in \mathbb R^n$ with $\delta>0$;

$B_\delta(S)=\cup_{x\in S} B_\delta(x)$ -- $\delta$-neighborhood of a set $S\subset\mathbb R^{n}$;

$\partial M$, $\overline M$ -- the boundary and the closure of a set $M\subset\mathbb R^n$, respectively; $\overline M= M\cup \partial M$;

$|S|$ -- the cardinality of a set $S$;

$\mathcal K$ --  the class of  continuous strictly increasing functions $\varphi:\mathbb R^+\to\mathbb R^+$ such that $\varphi(0)=0$;

$\nabla_xP(x^0,y^0)=\left.\frac{\partial P(x,y)}{\partial x}\right|_{x=x^0,y=y^0}$

$[f,g](x)$ -- the Lie bracket of vector fields $f,g:\mathbb R^n\to\mathbb R^n $ at a point $x\in\mathbb R^n$,   $[f,g](x)=L_fg(x)- L_gf(x)$, where  $ L_gf(x)=\lim\limits_{s\to0}\tfrac{f(x+sg(x))-f(x)}{s}$. In the case of differentiable $f$, $L_gf(x)=\frac{\partial f(x)}{\partial x}g(x)$.

\section{Related work}
\label{sec_related}

A number of efficient control design methods have been  developed in the literature with the emphasis on {\em special classes} of nonlinear systems, such as flat systems~\cite{Fliess95},
chained-form systems~\cite{Teel95,Li17}, unicycle- and car-like systems~\cite{Qu04,Mas07,De10,Sav15,Kap17}, manipulator models~\cite{Mac85,Gal04}, Chaplygin systems~\cite{Rey93}, etc.

For planning the motion of {\em general nonholonomic systems}, a  broad class of approaches is based on the application of Lie-algebraic techniques.
With this respect, an essential assumption is that the vector fields of system~\eqref{Sigma_a} together with their iterated Lie brackets span the whole tangent space at each point of the state manifold (H\"ormander's condition).
Several authors used this assumption to produce time-periodic control laws such that the trajectories of a nonholonomic system approximate the trajectories of an extended system.
 The papers~\cite{Sus91,Liu97}   exploited an unbounded sequence
of oscillating controls with unbounded frequencies for such an approximation in case of driftless systems.
The paper~\cite{KJ88} addresses the limit behavior of solutions of a control-affine system with input signals of magnitude $\varepsilon^{-\alpha}$ and frequency scaling $1/\varepsilon$ as $\varepsilon\to 0$. It is assumed that the primitives of input signals and their iterated primitives up to a certain order are bounded. Then it is shown that the limit behavior of the considered oscillating system is either defined by its drift term or by a linear combination of certain iterated Lie brackets, depending on the value of $\alpha$.
 In the paper~\cite{Bom13}, the averaged system as a differential inclusion is constructed for driftless control-affine systems with fast oscillating inputs. It is proved that an arbitrary solution of such a differential inclusion can be approximated by a family of solution of the original system when the oscillation frequency tends to infinity. This approximation result is also extended to the class of systems with drift under a time reparametrization and the assumption that the drift generates periodic dynamics.
 %
%
An overview of motion planning methods for nonholonomic systems is presented in the book~\cite{Jean14}. For nilpotent systems, exact solutions to the motion planning problem are proposed with the use of sinusoidal inputs. In general case, the local steering problem can be solved by constructing a nilpotent approximation under a suitable choice of privileged coordinates. Then the global steering algorithm is summarized in~\cite{Jean14} as a finite sequence of steps which steers the given nonholonomic system to an arbitrary small neighborhood of the target point.
The nilpotentization of a wheeled mobile robot model with a trailer is proposed in the paper~\cite{ARC21} for planing local maneuvers of this kinematic system. On the basis of solving the related sub-Riemannian problem, an algorithm for suboptimal parking has been implemented and tested for several robot configurations.

An algorithm for motion planning of kinematic models of nonholonomic systems in task-space is developed in~\cite{AS21} with the use of the Campbell--Baker--Hausdorff--Dynkin formula. The motion planning in task-space is treated in the sense of steering the system output to a neighborhood of the desired point. The proposed algorithm is illustrated with a unicycle and kinematic car examples.
A nonholonomic snake-like robot model with $m$ ($m \ge 3$) rigid links is considered in~\cite{ND21}.
The motion planning problem is treated there in the sense of generating a gait such that the origin of the snake's body moves along a given planar curve. This problem is solved by expressing the body velocity from the compatibility equation and reconstruction equation.

An interesting example of nonholonomic system with the growth vector (4,7) is studied in~\cite{JDCS20}. Such an example is a modification of the trident snake robot with three 1-link branches of variable length. A nilpotent approximation of this system is constructed, and the local optimal steering problem is analyzed by the Pontryagin maximum principle.
Controls for generating the motion in the direction of higher order Lie brackets were proposed in~\cite{Gau14,Gau15} for systems with two inputs.

A hybrid path planning method based on the combination of a high-level planner with a low-level controller performing in autonomous vehicle is described in~\cite{JIRS19}. The high-level planner ($D^*$ Lite planner) works on the discretized 2D workspace to produce a reference path such that at each step the robot model moves from a given cell to one of the eight neighboring cell which does not have an obstacle. The output of the high-level planner is collected as a set of waypoints ending at the goal, and the cost is the total length of the path. Then the low-level controller, running on the autonomous vehicle, provides control inputs to generate motion from the current state to the next waypoint. This path planning method is experimentally validated on a differential drive robot in rough terrain environments.

Stabilizing time-varying controls  were proposed in~\cite{ZuSIAM} for second degree nonholonomic systems (following the terminology used in~\cite{Laumond96}).
Unlike other publication in this area, the exponential convergence to the equilibrium was proved without the assumption that the frequencies of controls tend to infinity.
Besides, the paper~\cite{ZuSIAM}  presented a rigorous solvability analysis of the stabilization problem in the proposed class of controls.
For detailed reviews of other motion planning and stabilizing strategies we refer to~\cite{Kolman95,Bloch15,Hoy15}.
It has to be emphasized that, in spite of a large number of publications on nonholonomic motion planning, only particular results are available for the stabilization together with the obstacle avoidance.
 Even for static obstacles,  this problem was studied only for specific systems (see, e.g.,~\cite{Kod90,Rim92,Va12}).
 A general class of nonholonomic systems was considered in the paper~\cite{Ura17}, where a time-independent controller was constructed based on the gradient of a  potential function.
Note that such a result ensures only stability (but not asymptotic stability) property.
An algorithm  computing time-periodic feedback controls for approximating collision-free paths was presented in~\cite{Gur93},
however, no solvability issues  concerning the general collision avoidance problem have  been addressed in that paper.

For a class of driftless control-affine systems, the trajectory tracking problem was addressed in~\cite{Walsh94} under the assumption that the  target trajectory is feasible, i.e. satisfies the dynamical equations with some control inputs.
However,  to the best of our knowledge, there are no results available for the stabilization of  general classes of nonlinear control systems in a neighborhood of non-feasible curves or in domains with obstacles.


\section{Unified control framework for second degree nonholonomic systems}
\label{sec_unified}

In this section, we present the main idea of our control design scheme by considering the nonholonomic systems of degree 2, according to the classification of~\cite{Laumond96}.
The proposed control design provides a generic approach for stabilization and motion planning of underactuated driftless control-affine systems.

\subsection{Definitions and assumptions}\label{subs_defs}

To generate stabilizing control strategies, we will exploit sampling, similar to the approaches of~\cite{Clar97,ZuSIAM}.
With this respect, we introduce the following definition.
\begin{definition}[$\pi_\varepsilon$-solution]~\label{def_pi}
Consider a control system
$$
\dot x=f(x,u,t),\quad x\in D\subseteq \mathbb R^n,\,u\in\mathbb R^m,\,t\in\mathbb R,\,f:D\times\mathbb R^m\times\mathbb R\to\mathbb R^n,
$$
and assume  that
a feedback control is given in the form $u=h(a(x,t),t)$, $a:D\times\mathbb R\to\mathbb R^l$, $h: \mathbb R^l\times\mathbb R\to\mathbb R^m$.
For given $t_0\in\mathbb R$ and $\varepsilon>0$,  define a partition $\pi_\varepsilon$ of $[t_0,+\infty)$ into the intervals
$$
I_j=[t_j,t_{j+1}),\;t_j=t_0+\varepsilon j, \quad j=0,1,2,\dots \, .
$$
 A $\pi_\varepsilon$-solution of the considered closed-loop corresponding to
 the initial value $x^0\in\mathbb R^n$ is an absolutely continuous function  $x_\pi(t)\in D$, defined for $t\in[t_0,+\infty)$, which satisfies the initial condition  $x_\pi(t_0)=x^0$ and the differential equations
$$
\dot x_\pi(t)=f\big(x_\pi(t), h( a(x_\pi(t_j),t_j) ,t)  , t\big), \quad t\in I_j,\; \text{for each}\; j=0,1,2,\dots \, .
$$
\end{definition}


We will illustrate the relation between $\pi_\varepsilon$-solutions and classical solutions with examples in Section~\ref{sec_exmpls1}.



Before formulating basic results of this paper, we introduce the main assumptions on the state space $D$, vector fields $f_k$, and the potential function $P$ used in the gradient flow dynamics~\eqref{sys_grad}.

\begin{assumption}\label{as_f}
\quad\\
  The vector fields $f_k(x,t):D\times\mathbb R^+\to \mathbb R^n$ are twice continuously differentiable w.r.t. $x$, and $f_k$, $L_{f_{j}}f_k$ are continuously differentiable w.r.t. $t$, for all $j,k=\overline{1,m}$.

  Moreover, for any family of compact subsets $\widetilde{\mathcal D}_t\subset  D$, $t\ge 0$, there exist constants $M_f,L_{fx},L_{2f}>0$, $L_{ft},H_{fx},H_{ft}\ge0$ such that
  \begin{itemize}
    \item [\ref{as_f}.1)] $\|f_k(x,t)\|\le M_f$,
    \item [\ref{as_f}.2)] $\|f_k(x,t)-f_k(y,t)\|\le L_{fx}\|x-y\|,\,\Big\|\frac{\partial f_k(x,t)}{\partial t}\Big\|\le L_{ft},\,\|L_{f_{j}}f_k(x,t)\|\le L_{2f}$,
    \item [\ref{as_f}.3)]$\|L_{f_l}L_{f_{j}}f_k(x,t)\|\le H_{fx},\,\Big\|\frac{\partial (L_{f_{j}}f_k(x,t))}{\partial t}\Big\|\le H_{ft}$,
  \end{itemize}
  for all $t\ge0$, $x,y\in \widetilde{\mathcal D}_t$, $j,k,l=\overline{1,m}$.
\end{assumption}

Another important assumption is related to the controllability property of system~\eqref{Sigma}. As it has already been mentioned, in this section we focus on systems with the degree of nonholonomy 2, i.e. those whose vector fields together with their Lie brackets span the whole $n$-dimensional space.
\begin{assumption}\label{as_rank}
 \quad\\
  \begin{itemize}
    \item [\ref{as_rank}.1)] System~\eqref{Sigma} satisfies the bracket-generating condition of degree 2 in $D$, i.e. there exist sets of indices $S_1\subseteq \{1,2,...,m\}$,  $S_2\subseteq \{1,2,...,m\}^2$ such that $|S_1|+|S_2|=n$ and
\begin{equation}\label{rank}
\begin{aligned}
 {\rm span}\big\{f_{i}(x,t), [f_{j_1},f_{j_2}](x,t)\,|\,i\in S_1,(j_1,j_2)\in S_2\big\}=\mathbb{R}^n\,\text{ for all }t\ge0,x\in D.
\end{aligned}
\end{equation}
    \item [\ref{as_rank}.2)] For any family of compact subsets $\widetilde{\mathcal D}_t \subset  D$, ${t\ge0}$, there exists an ${M_F}>0$ such that
  $$
  \begin{aligned}
    \| \mathcal F^{-1}(x,t)\|\le {M_F}\text{ for all }t\ge0,\,x\in \widetilde{\mathcal D}_t,
  \end{aligned}
  $$
  where $\mathcal F^{-1}(x,t)$ is the inverse matrix for
  \begin{equation}\label{Fmatrix}
  \mathcal F(x,t)= \Big(\big(f_{j_1}(x,t)\big)_{j_1\in S_1}\ \ \big([f_{j_1},f_{j_2}](x,t)\big)_{(j_1,j_2)\in S_2}\Big).
\end{equation}
  \end{itemize}

\end{assumption}
It is important to note that the rank condition~\eqref{rank} implies nonsingularity of  the  $n\times n$ matrix   $\mathcal F(x,t)$ for all $t\ge0$, $x\in D$.

The next two assumptions describe properties of the potential function $P$ for the gradient-like system~\eqref{sys_grad}. 
\begin{assumption}\label{as_Pbounds}
\quad\\
The function $P: D\times\mathbb R^+ \to \mathbb R$ is twice continuously differentiable w.r.t.~$x$.
Moreover, for any family of compact subsets $\widetilde{\mathcal D}_t\subset  D$, $t\ge 0$, there exist constants $m_P\in\mathbb R$, $L_{Px}>0$, $L_{2Px},L_{2Pt},L_{Pt},H_{Px}\ge0$ such that
\begin{itemize}
  \item [\ref{as_Pbounds}.1)] $m_P\le P(x,t)$,
  \item [\ref{as_Pbounds}.2)] $\left\|\frac{\partial P(x,t)}{\partial x}\right\|\le L_{Px},\,\|P(x,t)-P(y,\tau)\|\le L_{Px}\|x-y\|+L_{Pt}\|t-\tau\|$,
  \item [\ref{as_Pbounds}.3)] $\|\nabla_xP(x,t)-\nabla_xP(y,\tau)\|\le L_{2Px}\|x-y\|+L_{2Pt}\|t-\tau\|$,
  \item [\ref{as_Pbounds}.4)] $\sum\limits_{i,j=1}^n\left\|\dfrac{\partial^2 P(x,t)}{\partial x_i\partial x_j}\right\|\le  H_{Px}$,
\end{itemize}
for all $t,\tau\ge0$, $x\in \widetilde{\mathcal D}_t$, $y\in \widetilde{\mathcal D}_\tau$.
\end{assumption}

To formulate the last assumption of this section, we introduce families of level sets for a function $P(x,t)$. Namely, given a constant $c\in\mathbb R$, we denote
$$
\mathcal L_t^{P,c} = \{x\in\mathbb D:P(x,t)\le c\},\;
{\mathcal L}_t^{\nabla P,c} = \{x\in\mathbb D:\|\nabla_xP(x,t)\|\le c\}\; \text{for}\; t\ge 0.
$$
\begin{assumption}\label{as_Psets}
\quad

For every  $x^0\in D$,  there exist $\lambda>0$ and $\rho>0$   such that,  for all
$t\ge t_0\ge 0$, $
 \mathcal L_t^{P,P(x^0,t_0)+\lambda}$ is
 non-empty, compact, convex set, and
 $$
   \mathcal L_t^{\nabla P,\rho}\subseteq  \mathcal L_t^{P,P(x^0,t_0)+\lambda}\subset  D.
  $$

\end{assumption}

\subsection{Convergence results}
\label{sec_basic}

Below we propose a universal control strategy which ensures the convergence of the trajectories of system~\eqref{Sigma} to the set of extremum points of a given function~$P$.
Suppose that the index sets $S_1$, $S_2$ and the matrix ${\cal F}(x,t)$ are described in Assumption~\ref{as_rank},
then we parameterize the controls as
\begin{equation}\label{cont}
\begin{aligned}
u_k= u^\varepsilon_k(a(x,t),t)=&\sum_{i\in S_1} a_{i}(x,t)\delta_{ki} \\
&+\varepsilon^{-\tfrac{1}{2}}\sum_{(j_1,j_2)\in S_2} \sqrt{|a_{j_1j_2}(x,t)|}\phi^{(k,\varepsilon)}_{j_1j_2}(t),\, k = \overline{1,m}.
\end{aligned}
  \end{equation}
Here the column vector $a(x,t)=\big(a_{i_1}(x,t)\big|_{i_1\in S_1}, a_{j_1j_2}(x,t)\big|_{(j_1,j_2)\in S_2}\big)^\top\in\mathbb R^n$ is obtained from
\begin{align}
a(x,t)&=- \gamma \mathcal F^{-1}(x,t) \nabla_x P(x,t)\label{a1},
\end{align}
and the oscillating components are
\begin{equation}
 \phi^{(k,\varepsilon)}_{j_1j_2}(t)=2\sqrt{\pi \kappa_{j_1j_2}}\Big(\delta_{kj_1}{\rm sign}(a_{j_1,j_2}(x,t))\cos{\frac{2\pi \kappa_{j_1j_2}}{\varepsilon}}t
 +\delta_{kj_2}\sin{\frac{2\pi \kappa_{j_1j_2}}{\varepsilon}}t\Big),
 \label{dithers}
\end{equation}
where $\kappa_{j_1j_2}\in\mathbb N$ are pairwise distinct numbers,
$\gamma>0$ is a control gain, and $\varepsilon>0$ is a small parameter.

The first result of this section is as follows.

  \begin{lemma}~\label{lem_step1}
Let Assumptions~\ref{as_f}--\ref{as_Psets} be satisfied for system~\eqref{Sigma} with a function $P(x,t)$.
   Then there exist a $\bar\gamma>0$ and $\bar\varepsilon:[\bar \gamma,+\infty)\to {\mathbb R}_{>0}$ such that, for any  $\gamma\ge \bar\gamma$ and any $\varepsilon\in(0,\bar\varepsilon(\gamma)]$, the $\pi_\varepsilon$-solution $x_\pi(t)$ of system~\eqref{Sigma} with the controls $u_k=u_k^\varepsilon(a(x,t),t)$ given by~\eqref{cont}--\eqref{dithers}
   and the initial data  $x_\pi(t_0)=x^0\in D$, $t_0\ge 0$ is well-defined and $x_\pi(t)\in \mathcal L_t^{P,P(x^0,t_0)+\lambda}$ for all $t\ge t_0$, and there exists a $T\ge 0$ such that
  $$
  P(x_\pi(t),t)\le \sup_{t\ge t_0+ T} \sup\limits_{\xi\in \mathcal L^{\nabla P,\rho}_{t}}P(\xi,t)\text{ for all }t\ge t_0+T,
  $$ 
  where  $\lambda$, $\rho$ are positive numbers from Assumption~\ref{as_Psets}.
    \end{lemma}

      The proof is in  Appendix~\ref{proof_step1}.



  In the case of time-independent function $P(x)$ and vector fields $f_k(x)$, it is possible to prove a stronger result under milder assumptions.
   Let us denote   the set of local minima of the function $P$ by
 $$S^*_{{\min}}=\{x^*\in D:\text{ there exists }r > 0\text{ s.t. }P(x)\ge P(x^*)\text{ for all }x\in B_r(x^*)\}.$$
  The following theorem holds for the system
  \begin{equation}
\dot x = \sum_{k=1}^m f_k (x)u_k,\quad x\in D\subseteq\mathbb R^n, \,u\in\mathbb R^m.
\label{Sigma_x}
\end{equation}

  \begin{theorem}\label{thm_step1Px}
\quad

      Given system~\eqref{Sigma_x}, let $f_k\in C^2(D;\mathbb R^n)$ satisfy Assumption~\ref{as_rank} in a {domain} $D\subseteq\mathbb R^n$, and let a function $P\in C^2(D;\mathbb R)$ be such that its level sets  $\mathcal L^{P,P(x^0)}=\{x\in D: P(x)\le P(x^0)\}$ are compact for all $x^0\in D$.

      Then for any $\gamma>0$ there exists an $\bar\varepsilon>0$   such that, for any  $\varepsilon\in(0,\bar \varepsilon]$, the $\pi_\varepsilon$-solution $x_\pi(t)$ of system~\eqref{Sigma_x} with the controls
      $u_k=u_k^\varepsilon(a(x),t)$ given by~\eqref{cont}--\eqref{dithers} and the initial data $t_0\ge 0$, $x_\pi(t_0)=x^0\in   D$ is well-defined, and  satisfies the following property:
\begin{equation}\label{Sp}
P(x_\pi(t))\to \alpha^*\in {S^*_{P_{\min}}} \text{ as }t\to+\infty,
\end{equation}
provided that $x^0\notin \{x\in D:\nabla P(x)=0\}\setminus S^*_{{\min}}$. Here
$$S^*_{P_{\min}}=\{P^*\in[m_P,P(x^0)]:\text{ there exists }x^*\in S^*_{{\min}}\text{ such that }P^*=P(x^*)  \}.$$
  \end{theorem}
    The proof of the asymptotic convergence of $P(x_\pi(t))$ to the set of critical values of $P$ can be found in~\cite{ZG17a}. {More strict property~\eqref{Sp} follows from the fact that, for small enough $\varepsilon$, $P(x^0)\ge P(x_\pi(t_0+\varepsilon))\ge P(x_\pi(t_0+2\varepsilon))\ge\dots $ and the uniqueness of the solutions of system~\eqref{Sigma_x} with the controls $u_k=u_k^\varepsilon(a(x),t)$ and the initial data $t_0\ge 0$, $x(t_0)=x^0\in   D$.}


The approximate convergence of a time-varying function $P$ to its minimal value can be proved under an additional requirement, which also allows to estimate the convergence rate:

    \begin{theorem}~\label{thm_step1Pxt}

    Let Assumptions~\ref{as_f}--\ref{as_Pbounds} be satisfied for system~\eqref{Sigma} with a function $P(x,t)$, and let $\rho>0$ be such that $\emptyset \neq \mathcal L_t^{P,m_P+\rho}\subset D$ for all $t\ge0$.
     Assume moreover that, for any family of compact subsets $\widetilde{\mathcal D}_t\subset  D$, ${t\ge0}$, there exists a $\mu>0$ and $\nu\ge0$ such that
  \begin{equation}\label{Pgrad}
    \|\nabla_xP(x,t)\|^2\ge \mu (P(x,t)-m_P)^\nu\text{ for all }x\in \widetilde{\mathcal D}_t,\,t\ge0.
  \end{equation}
  Then for any $\gamma^*>0$ there is a $\bar\gamma>\gamma^*$ such that, for any  $\gamma>\bar\gamma$ and $\varepsilon\in(0,\bar\varepsilon)$ ($\bar \varepsilon> 0$ depends on $\gamma$), the $\pi_\varepsilon$-solution $x_\pi(t)$ of system~\eqref{Sigma} with the controls $u_k=u_k^\varepsilon(a(x,t),t)$
  given by~\eqref{cont}--\eqref{dithers} and the initial data $t_0\ge 0$, $x_\pi(t_0)=x^0\in \mathcal  D_{t_0}$ is well-defined, and
    satisfies one of the following properties:
  \begin{itemize}
    \item [I)] If $\nu=1$, then
  $$
P(x_\pi(t),t)-m_P\le (P(x^0,t_0)-m_P)e^{-\mu\gamma^*(t-t_0-\varepsilon)}+\rho\text{ for all }t\ge t_0.$$
    \item [II)] If $\nu>1$, then
$$
P(x_\pi(t),t)-m_P\le \big((P(x^0,t_0)-m_P)^{1-\nu}+\mu\gamma^*(\nu-1)(t-t_0-\varepsilon)\big)^{\frac{1}{1-\nu}}+\rho,\, t\ge t_0.
$$
  \end{itemize}
    \end{theorem}
The proof is in Appendix~\ref{proof_step1Pxt}.
\begin{remark}
As it follows from the proof of Theorem~\ref{thm_step1Pxt}, it suffices to take
$$
\bar\gamma=\gamma^*+\frac{2^{2\nu}}{\rho^\nu\mu}\left(L_{Pt}+L_{Px}c_{R1}(\sqrt{{M_F} L_{Px}}+{H_{Px}}c_{R1}{M_F}\bar\varepsilon)\right),
$$
where $c_{R1}=\frac{L_{ft}}{2}+\frac{H_{ft}}{6}\sqrt{{M_F} L_{Px}}$. Obviously, one may put $\bar\gamma=\gamma^*+\dfrac{2^{2\nu}}{\rho^\nu\mu}L_{Pt}$ if the vector fields of system~\eqref{Sigma} are time-independent, and
$\bar\gamma=\gamma^*$ if, additionally, the function $P$ does not depend on $t$.
\end{remark}

    \begin{corollary}~\label{cor1_step1}
      Assume that the constants required in Assumptions~\ref{as_f}--\ref{as_Pbounds} (and in~\eqref{Pgrad}) exist for all $x\in \mathcal L_t^{P,P(x^0,t_0)}$, $x^0\in D$, $t_0\ge0$. Then the assertions of Lemma~\ref{lem_step1} (Theorem~\ref{thm_step1Pxt}) remain valid even if the level sets of the function $P(x,t)$ are not compact.

      Similarly, if the functions $f_k(x)$ are globally Lipschitz in $\mathcal L^{P,P(x^0)} $, for any $x^0\in D$, the functions $f_k(x)$, $L_{f_{j}}f_k(x,t)$, $L_{f_l}L_{f_{j}}f_k(x)$, $\| \mathcal F^{-1}(x)\|$, $\frac{\partial P(x)}{\partial x}$, $\frac{\partial^2 P(x)}{\partial x^2}$ are bounded, and the function $P(x)$ is bounded from below for all $x\in \mathcal L^{P,P(x^0)} $, $x^0\in D$, then the assertion of Theorem~\ref{thm_step1Px}  remains valid even if the level sets of the function $P(x)$ are not compact.
    \end{corollary}
    \begin{corollary}~\label{cor1_step2}

      Let the conditions of Theorem~\ref{thm_step1Px} be satisfied.  Furthermore, assume that for any compact subset $\widetilde{\mathcal D}\subset  D$ there exist a $\mu>0$ and $\nu\ge 1$  such that
$
    \|\nabla P(x)\|^2\ge \mu (P(x)-m_P)^\nu\text{ for all }x\in \widetilde{\mathcal D},
$
  where $m_P$ is defined in Assumption~\ref{as_Pbounds}.1.

   Then for any $\gamma>\gamma^*>0$ there exists an $\bar\varepsilon>0$  such that, for any $\varepsilon\in(0,\bar\varepsilon)$, the $\pi_\varepsilon$-solution $x_\pi(t)$ of system~\eqref{Sigma_x} with the controls $u_k=u_k^\varepsilon(a(x),t)$ given by~\eqref{cont}--\eqref{dithers}
   and the initial data $t_0\ge 0$, $x_\pi(t_0)=x^0\in   D$ is well-defined, and satisfies one of the following properties:
     \begin{itemize}
    \item [I)] If $\nu=1$, then
  $$
P(x_\pi(t))-m_P\le (P(x^0)-m_P)e^{-\mu\gamma^*(t-t_0-\varepsilon)}\text{ for all }t\ge t_0.$$
    \item [II)] If $\nu>1$, then
$$
P(x_\pi(t))-m_P\le \big((P(x^0)-m_P)^{1-\nu}+\mu\gamma^*(\nu-1)(t-t_0-\varepsilon)\big)^{\frac{1}{1-\nu}}\text{ for all }t\ge t_0.
$$
  \end{itemize}
    \end{corollary}
    These results follow from the proofs of Lemma~\ref{lem_step1} and Theorem~\ref{thm_step1Pxt}.

Lemma~\ref{lem_step1} and Theorem~\ref{thm_step1Px} give rise to several important results applicable to more specific control problems.

\subsection{Stabilization problem}
\label{sec_stab}

In this section, we consider a classical control problem of finding control laws which ensure the asymptotic stability of a point $x=x^*\in D$ for system~\eqref{Sigma_x}.
\begin{problem}[Stabilization problem]\label{prob_stab}
Given system~\eqref{Sigma_x} and a point $x^*\in D$, the goal is to construct a feedback control of the form~\eqref{cont}--\eqref{dithers} ensuring the asymptotic stability of $x^*$ for the corresponding closed-loop system.
\end{problem}



To solve Problem~1, we apply the results of Section~\ref{sec_basic} with a Lyapunov-like function $P(x)$:
\begin{theorem}\label{thm_step1stab}
\quad

      Given system~\eqref{Sigma_x} with $f_k\in C^2(D;\mathbb R^n)$ satisfying Assumption~\ref{as_rank} in a {domain} $D\subseteq\mathbb R^n$ and
        a point $x^*\in D$, let a function $P\in C^2(D;\mathbb R)$ satisfy the following conditions:
      \begin{itemize}
        \item \ref{thm_step1stab}.1) there exist functions $w_{11},w_{12}\in \mathcal K$ such that $\{x\in \mathbb R^n: \|x-x^*\|\le w_{11}^{-1}\big(P(x^0)-m_P\big)\}\subset D$ for all $x^0\in D$, and
            $$
          w_{11}(\|x-x^*\|)\le  P(x)-m_P\le w_{12}(\|x-x^*\|)\text{ for all }x\in D;
            $$
        \item \ref{thm_step1stab}.2) $\|\nabla P(x)\|=0$ if and only if $x=x^*$, and there exists a function $w_2\in \mathcal K$ such that
        $$
            \|\nabla P(x)\|\le w_2(\|x-x^*\|)\text{ for all }x\in D.
            $$
      \end{itemize}
Then for any $\gamma>0$ there exists an $\bar\varepsilon>0$  such that  the point $x^*$ is asymptotically stable for system~\eqref{Sigma_x} with the controls $u_k=u_k^\varepsilon(a(x),t)$ given by~\eqref{cont}--\eqref{dithers} and any  $\varepsilon\in(0,\bar \varepsilon)$,
provided that the solutions of the closed-loop system~\eqref{Sigma_x},~\eqref{cont}--\eqref{dithers} are defined in the sense of Definition~\ref{def_pi}.
  \end{theorem}

  The proof of this  theorem is based on the proofs of Lemma~\ref{lem_step1} and Theorem~\ref{thm_step1Px} (see Appendix~\ref{proof_step1stab}). The following result directly follows from Theorem~\ref{thm_step1stab} and Corollary~\ref{cor1_step2}:

\begin{corollary}\label{thm_step1rate}
\quad

      Given system~\eqref{Sigma_x} with $f_k\in C^2(D;\mathbb R^n)$ satisfying Assumption~\ref{as_rank} in a {domain} $D\subseteq\mathbb R^n$ and
        a point $x^*\in D$, let a function $P\in C^2(D;\mathbb R)$ satisfy the following conditions:
      \begin{itemize}
        \item \ref{thm_step1rate}.1) there exist constants $\omega_{11},\omega_{12},v_1,v_2>0$ such that
            $$
          \omega_{11}\|x-x^*\|^{v_1}\le  P(x)-m_P\le \omega_{12}\|x-x^*\|^{v_2}\text{ for all }x\in D;
            $$
        \item \ref{thm_step1rate}.2) there exist constants $\mu_1,\mu_2>0$ and $\nu_1,\nu_2\ge 1$ such that
        $$
       \mu_1 (P(x)-m_P)^{\nu_1}\le \|\nabla P(x)\|^2\le \mu_2 (P(x)-m_P)^{\nu_2}\text{ for all }x\in D.
        $$
      \end{itemize}
Then for any $\gamma>0$ there exists an $\bar\varepsilon>0$  such that  the point $x^*$ is asymptotically stable for the closed-loop system~\eqref{Sigma_x} with the controls $u_k=u_k^\varepsilon(a(x),t)$ given by~\eqref{cont}--\eqref{dithers}
and any  $\varepsilon\in(0,\bar \varepsilon)$, 
provided that the solutions of the closed-loop system are defined in the sense of Definition~\ref{def_pi}. Moreover,
 \begin{itemize}
    \item [I)] If $\nu_1=1$, then $x^*$ is exponentially stable; namely, for any $\gamma>\gamma^*>0$, there exists an $\varepsilon>0$ such that
  $$
\|x_\pi(t)-x^*\|\le \beta\|x^0-x^*\|^{\frac{v_2}{v_1}}e^{-\frac{\mu_1\gamma^*}{v_1}(t-t_0-\varepsilon)}\text{ for all }t\ge t_0,$$
where $\beta=\left(\frac{\omega_{12}}{\omega_{11}}\right)^{\frac{1}{v_1}}$.
    \item [II)] If $\nu_1>1$, then $x^*$ is polynomially stable, namely, for any $\gamma^*>0$ and $\gamma>\gamma^*$ there exists an $\varepsilon>0$ such that
$$
\|x_\pi(t)-x^*\|\le \left(\beta_1\|x^0-x^*\|^{v_2(1-\nu_1)}+\beta_2(t-t_0-\varepsilon)\right)^{\frac{1}{v_1(1-\nu_1)}}\text{ for all }t\ge t_0,
$$
where $\beta_1=\left(\dfrac{\omega_{12}}{\omega_{11}}\right)^{1-\nu_1}$, $\beta_2=\dfrac{\mu_1\gamma^*(\nu_1-1)}{\omega_{11}^{1-\nu_1}}$.
  \end{itemize}
  \end{corollary}
In particular, to exponentially stabilize system~\eqref{Sigma_x} at $x^*$, one can simply put
$$P(x)=\|x-x^*\|^2.$$
The above-stated decay rate estimates are illustrated with numerical examples in Section~\ref{sec_exmpls1_stab}.

\begin{remark}
   It is interesting to note that for the degree 1 nonholomonic systems, i.e. for the case $m=n$, $S_1=\{1,\dots,n\}$, the proposed stabilizing controls are time-invariant functions $$u_i^\varepsilon(x,t)=u_i(x)=-  \big(f_1(x)\  f_2(x)\ \dots\ f_n(x)\big)^{-1} \nabla P(x),$$
    which is  the classical control design for stabilization of fully-actuated driftless control-affine systems.
 \end{remark}

 \begin{remark}
   The proposed control algorithm~\eqref{cont}--\eqref{dithers} significantly simplifies the stabilizing control design procedure introduced in~\cite{ZuSIAM} and makes it possible to express control coefficients explicitly without solving a cumbersome system of algebraic equations.
 \end{remark}

  \subsection{Trajectory tracking problem}
  \label{sec_tracking}

  The proposed control design procedure with a time-varying function $P(x,t)$ can be used for ensuring the motion of system~\eqref{Sigma} along desirable curves. Note that we consider arbitrary continuous curves $x^*(t)$ which may not be feasible for system~\eqref{Sigma}. Consequently, we consider a relaxed problem statement for the \emph{approximate} trajectory tracking as follows:
  \begin{problem}[Trajectory tracking problem]\label{prob_tracking}

Given system~\eqref{Sigma}, a continuous curve $x^*:\mathbb R^+\to D$, and a constant $\rho>0$, the goal is to construct a feedback law ensuring the attractivity of the family of sets
\begin{equation}
\mathcal L^\rho_t=\{x\in D: \|x-x^*(t)\|\le\rho\}_{t\ge0}.
\label{Lsets}
\end{equation}
for the corresponding closed-loop system.
\end{problem}

Note that attracting (locally/globally pullback attracting) families of time-varying sets have been studied in the paper~\cite{langa2002} for non-autonomous systems of ordinary differential equations. Here we treat this notion in the sense of $\pi_\varepsilon$-solutions (Definition~\ref{def_pi}) for system~\eqref{Sigma} with control inputs. To be precise, we introduce the following definition.

\begin{definition}[Attracting family of sets in the sense of $\pi_\varepsilon$-solutions]~\label{def_attr}
Let a feedback control of the form~\eqref{cont}--\eqref{dithers} be given, and let $\rho>0$.
We call the family of sets~\eqref{Lsets} {\em attracting} for the closed-loop system~\eqref{Sigma},~\eqref{cont}--\eqref{dithers}, if
there exist $\Delta>0$, $\bar \gamma>0$, and $\bar \varepsilon:[\bar \gamma,+\infty)\to \mathbb R_{>0}$ such that, for any $t_0\ge 0$, $x^0\in B_{\Delta}(\mathcal L^{\rho}_{t_0})\cap D$, $\gamma\ge \bar \gamma$, $\varepsilon\in (0,\bar \varepsilon(\gamma)]$,
the corresponding $\pi_\varepsilon$-solution $x_\pi(t)$ satisfying the initial condition $x_\pi(t_0)=x^0$ is well-defined and
$$
{\rm dist}(x_\pi(t),\mathcal L^{\rho}_{t}) \to 0\quad\text{as}\;\;t\to + \infty.
$$
\end{definition}

Based on Theorem~\ref{thm_step1Pxt}, we are in a position to state sufficient conditions for the solvability of Problem~\ref{prob_tracking}.
\begin{theorem}\label{thm_step1track}
\quad

Given system~\eqref{Sigma}, a continuous curve $x^*: \mathbb R^+\to D$, and a function $P:D\times\mathbb R^+\to\mathbb R$,  let Assumptions~\ref{as_f}--\ref{as_Psets} be satisfied, and assume that the following conditions hold:
   \begin{itemize}
        \item \ref{thm_step1track}.1) there exist constants $\omega_{11},\omega_{12},v_1,v_2>0$ such that
            $$
          \omega_{11}\|x-x^*(t)\|^{v_1}\le  P(x,t)-m_P\le \omega_{12}\|x-x^*(t)\|^{v_2}\text{ for all }t\ge 0,\,x\in D;
            $$
        \item \ref{thm_step1track}.2) there exist constants $\mu_1,\mu_2>0$ and $\nu_1,\nu_2\ge 1$ such that
        $$
       \mu_1 (P(x,t)-m_P)^{\nu_1}\le \|\nabla_x P(x,t)\|^2\le \mu_2 (P(x,t)-m_P)^{\nu_2}\text{ for all }t\ge 0,\,x\in D.
        $$
      \end{itemize}
  Then, for any $\rho>0$, the family of sets
  $\mathcal L^{\rho}_t=\{x\in D: \|x-x^*(t)\|\le\rho\}_{t\ge0}$
is attracting for the closed-loop system~\eqref{Sigma} with the controls $u_k=u_k^\varepsilon(a(x,t),t)$ given by~\eqref{cont}--\eqref{dithers}
in the sense of Definition~\ref{def_attr}.
Moreover, one of the following assertions holds for any $\gamma>\gamma^*\ge \bar \gamma$, $\varepsilon\in (0,\bar \varepsilon(\gamma)]$, and $x^0\in B_{\Delta}(\mathcal L^{\rho}_{t_0})\cap D$:
 \begin{itemize}
    \item [I)] if $\nu_1=1$, then   $\{\mathcal L^{\rho}_t\}_{t\ge0}$ is exponentially attractive, i.e.
  $$
\|x_\pi(t)-x^*\|\le \beta\|x^0-x^*\|^{\frac{v_2}{v_1}}e^{-\frac{\mu_1\gamma^*}{v_1}(t-t_0-\varepsilon)}+\rho \text{ for all }t\ge t_0,$$
where $\beta=\left(\frac{\omega_{12}}{\omega_{11}}\right)^{\frac{1}{v_1}}$;
    \item [II)] if $\nu_1>1$, then   $\{\mathcal L^{\rho_1}_t\}_{t\ge0}$ is polynomially attractive, i.e.
$$
\|x_\pi(t)-x^*\|\le \left(\beta_1\|x^0-x^*\|^{v_2(1-\nu_1)}+\beta_2(t-t_0-\varepsilon)\right)^{\frac{1}{v_1(1-\nu_1)}}+\rho \text{ for }t\ge t_0,
$$
where $\beta_1=\left(\dfrac{\omega_{12}}{\omega_{11}}\right)^{1-\nu_1}$ and $\beta_2=\dfrac{\mu_1\gamma^*(\nu_1-1)}{\omega_{11}^{1-\nu_1}}$.
\end{itemize}
  \end{theorem}

The proof is similar to the proof of Theorem~\ref{thm_step1stab}.

\begin{corollary}\label{cor_tracking}
\quad

  Given system~\eqref{Sigma_x} with $f_k\in C^2(D;\mathbb R^n)$ satisfying Assumption~\ref{as_rank} in a domain $D\subseteq\mathbb R^n$,
  let a curve $x^*:\mathbb R^+ \to D$ be Lipschitz continuous such that $B_\delta(x^*(t))\subset D$ for all $t\ge 0$ with some $\delta>0$.

  Then, for any $\rho>0$,
  the family of sets
  $\mathcal L^{\rho}_t=\{x\in D: \|x-x^*(t)\|\le\rho\}_{t\ge0}$
is (exponentially) attracting for the closed-loop system~\eqref{Sigma_x} with the controls $u_k=u_k^\varepsilon(a(x,t),t)$ given by~\eqref{cont}--\eqref{dithers}
in the sense of Definition~\ref{def_attr}.
  \end{corollary}

The above result has been proved in~\cite{GZ19} for continuously differentiable $x^*(t)$ with bounded first derivative.

  \subsection{Obstacle avoidance problem}
  \label{sec_obstacle}

Another important problem which can be solved by the proposed approach is generating collision-free motion of system~\eqref{Sigma_x} in environments with obstacles.
To formulate such problem, assume that the set $D$ is represented as a closed bounded domain with ``holes'', i.e.
$$
D=\mathcal W\setminus\bigcup_{j=1}^N\mathcal O_j,
$$
where $\mathcal W\subset \mathbb R^n$ is a closed bounded domain (workspace), and $\mathcal O_1,\mathcal O_2,..., \mathcal O_N\subset \mathcal W$ are open domains (obstacles).
The resulting set  $D$ is supposed to be \emph{valid}~\cite{Kod90}, i.e.
$\displaystyle
\overline{\mathcal O_i} \subset {\rm int}\, \mathcal W$ and $
\displaystyle\overline{\mathcal O_i} \cap \overline{\mathcal O_j} = \emptyset\;\;\text{if}\; \neq j$,
for all $i,j\in\{1,\dots,N\}$.
\begin{problem}[Obstacle avoidance problem]
Given system~\eqref{Sigma_x}, an initial point $x^0\in {\rm int}\, D$ and a destination point $x^*\in {\rm int}\, D$, the goal is to construct a feedback control such that the corresponding solution $x(t)$ of the closed-loop system~\eqref{Sigma_x}
with the initial data $x(0)=x^0$ satisfies the conditions:
\begin{itemize}
  \item collision-free motion: $ x(t)\in {\rm int}\, D\text{ for all }t\ge 0;$
  \item convergence to the target: $x(t)\to x^* \text{ as } t\to +\infty$.
\end{itemize}
\end{problem}
  As it is implied by Theorem~\ref{thm_step1Px}, the above problem can be solved by  the controls $u_k=u_k^\varepsilon(a(x,t),t)$ from~\eqref{cont}--\eqref{dithers} with a proper function $P\in C^2(D;\mathbb R)$ being such that its level sets  $\mathcal L^{P,P(x^0)}=\{x\in\mathbb R^n: P(x)\le P(x^0)\}$ are compact and $\mathcal L^{P,P(x^0)}\subset D$ for all $x^0\in D$ (see also~\cite{GZ18}). There is a broad range of potential functions ensuring collision-free motion for specific classes of systems, see, e.g.~\cite{Pat18}. Some of those functions can be used under our control-design framework for general classes of nonholonomic systems.
   As possible candidates for the function $P$, one can consider, e.g., the following:
   \begin{itemize}
     \item \emph{Navigation functions}. According to~\cite{Pat18}, a map $P\in C^2(D;[0,1])$ defined on a compact connected analytic manifold $D$ with boundary is a \emph{navigation function}, if it is: 1) polar at $x^*\in{\rm int} D$, i.e. has a unique minimum at $x^*$; 2)~Morse, i.e. its critical points on $D$ are nondegenerate; 3) admissible, i.e. all  boundary components have the same maximal value, namely $\partial D= P^{-1}(1)$.

         In particular, if $\mathcal W=\{x\in\mathbb R^n:\varphi_0(x)\ge0\}$ and $\mathcal O_i=\{x\in\mathbb R^n:\varphi_i(x)<0\}$, $i=\overline{1,N}$, with  convex functions  $\varphi_0,\varphi_i\in C^2(\mathbb R^n;\mathbb R)$, then the navigation function can be taken in the form
         \begin{equation}\label{P_nav}
           P(x)=\frac{\|x-x^*\|^2}{\big(\|x-x^*\|^{2K}+\varphi(x)\big)^{\frac{1}{K}}},\,\varphi(x)=\prod_{i=0}^N\varphi_i(x),
         \end{equation}
         provided that $K$ is large enough and, for all $x\in \partial O_i$, $i=\overline{1,N}$, $$\dfrac{\nabla\varphi_i(x)^\top(x-x^*)}{\|x-x^*\|^2}<c_{\varphi}^i,$$ where $c_{\varphi}^i$ is the minimal eigenvalue of the Hessian of $\varphi_i(x)$ (see~\cite{Pat18} for more details).
     \item \emph{Artificial potential fields}, which represent a combination of attractive and repulsive potential fields. In particular, one can take~\cite{Kha86}:
     \begin{equation}\label{P_art1}
       P(x)=\left\{
    \begin{aligned}
     &  \|x-x^*\|^2+K\left(\frac{1}{\varphi(x)}-\frac{1}{\varphi(\xi)}\right)^2&\text{ if }\varphi(x)\le \varphi(\xi),\\
     & K\|x-x^*\|^2&\text{ if }\varphi(x)> \varphi(\xi),
    \end{aligned}
       \right.
     \end{equation}
     where $K$ is a positive constant gain, $\xi$ belongs to a neighborhood of obstacles (see~\cite{Kha86} for more details). Another function of such type was proposed in~\cite{Va08}:
\begin{equation}\label{P_art2}
  P(x)= \|x-x^*\|^2\left(1+\frac{K}{\varphi(x)}\right),\quad K>0.
\end{equation}
   \end{itemize}

  We expect that a similar approach can be applied to time-varying navigation functions (or time-varying artificial potential fields) with the use of Lemma~\ref{lem_step1},
  i.e. in cases where either obstacles or destination point  are ``moving'', i.e. are given by time-dependent functions. We will illustrate this  claim via numerical examples in Section~\ref{sec_exmpls1}.
  However, the analysis of properties of such functions requires a separate study which we leave for future work.

  \section{Examples}
  \label{sec_exmpls1}

In this section, we will demonstrate the proposed control design approach on the mathematical model of a unicycle, which is a well-known example with the degree of nonholonomy 2.
The equations of motion have the form~\eqref{Sigma_x} with $n=3$, $m=2$, $f_1(x)=\big(\cos x_3,\sin x_3, 0\big)^\top$, $f_2(x)=\big(0,0,1\big)^\top$:
\begin{equation}\label{ex_uni}
  \begin{aligned}
    &\dot x_1=u_1\cos x_3,\\
    &\dot x_2=u_1\sin x_3,\\
    &\dot x_3=u_2.
  \end{aligned}
\end{equation}
Here $(x_1,x_2)$ denote the coordinates of the contact point of the unicycle wheel, $x_3$ is the angle between the wheel and the $x_1$-axis, $u_1$ and $u_2$ control the forward and the angular velocity, respectively.
It is easy to see that the vector fields of system~\eqref{ex_uni} satisfy Assumptions~\ref{as_f}--\ref{as_rank} in $D = \mathbb R^3$. In particular, Assumption~\ref{as_rank} holds with the set of indices $S_1=\{1,2\}$, $S_2=\{(1,2)\}$:
$$
 {\rm span}\big\{f_{1}(x),f_2(x), [f_{1},f_{2}](x)\big\}=\mathbb{R}^3 \;\text{ for all }x\in D=\mathbb R^3,
$$
so that the matrix
$$
\mathcal F(x)=\left(f_{1}(x)\ f_2(x)\ [f_{1},f_{2}](x)\right)
=
\left(
  \begin{array}{ccc}
    \cos x_3 & 0 & \sin x_3 \\
    \sin x_3 & 0 & -\cos x_3 \\
    0 & 1 & 0 \\
  \end{array}
\right)
$$
is nonsingular in $D$, and the corresponding inverse matrix
\begin{equation}\label{F1_uni}
  \mathcal F^{-1}(x)=
\left(
  \begin{array}{ccc}
    \cos x_3 &  \sin x_3 & 0 \\
   0 & 0 & 1 \\
     \sin x_3 & - \cos x_3 & 0 \\
  \end{array}
\right)
\end{equation}
has bounded norm for all $x\in D$.

According to the proposed control laws~\eqref{cont}, we take
\begin{equation}\label{cont_uni}
\begin{aligned}
&u_1= u^\varepsilon_1(a(x,t),t)=a_{1}(x,t) +2\sqrt{\frac{\pi|a_{12}(x,t)|}{\varepsilon}}\ {\rm sign}(a_{12}(x,t))\cos{\frac{2\pi t}{\varepsilon}},\\
&u_2= u^\varepsilon_2(a(x,t),t)=a_{2}(x,t) +2\sqrt{\frac{\pi|a_{12}(x,t)|}{\varepsilon}}\sin{\frac{2\pi t}{\varepsilon}}.
\end{aligned}
\end{equation}
In the above formulas, $\kappa_{12}$ is taken to be equal 1, and the vector of state-dependent coefficients $a(x,t)$ is defined by~\eqref{a1}:
$$
a(x,t)=\left(a_1(x,t)\ a_2(x,t)\ a_{12}(x,t)\right)^\top=-\gamma\mathcal F^{-1}(x,t)\nabla_x P(x,t),
$$
where $\gamma>0$ and $\varepsilon>0$ are control parameters, the matrix $\mathcal F^{-1}(x,t)$ is given by~\eqref{F1_uni}, and $P\in C^2(D\times\mathbb R;\mathbb R)$. Thus,
$$
\begin{aligned}
&a_1(x,t)=-\gamma\left(\frac{\partial P(x,t)}{\partial x_1}\cos x_3+\frac{\partial P(x,t)}{\partial x_2}\sin x_3\right),\\
&a_2(x,t)=-\gamma\frac{\partial P(x,t)}{\partial x_3},\\
&a_{12}(x,t)=-\gamma\left(\frac{\partial P(x,t)}{\partial x_1}\sin x_3-\frac{\partial P(x,t)}{\partial x_2}\cos x_3\right).
\end{aligned}
$$
Next, we will illustrate the behavior of solutions to system~\eqref{ex_uni},~\eqref{cont_uni} with different functions $P(x,t)$, depending on the control goal.
As it has been mentioned in Subsection~\ref{subs_defs}, the obtained control scheme can be used within the framework of sampling in the sense of Definition~\ref{def_pi}, and for classical solutions as well.
In the simulations below, we depict the trajectories of system~\eqref{ex_uni} with both types of solutions of the closed-loop system.

\subsection{Stabilization problem}\label{sec_exmpls1_stab}

We start with Problem~\ref{prob_stab} considered in Section~\ref{sec_stab}.
To exponentially stabilize system~\eqref{ex_uni} at an arbitrary $x^*\in\mathbb R^3$, one can take the simple quadratic function
\begin{equation}\label{P_stab_uni}
P(x)=\|x-x^*\|^2.
\end{equation}
According to Corollary~\ref{thm_step1rate}.I, the following decay rate estimate holds:
$$
\|x_\pi(t)-x^*\|\le \|x^0-x^*\|e^{-2\gamma(t-\varepsilon)}\text{ for all }t\ge 0.
$$
Fig.~\ref{fig_uni_stab} shows the trajectories of system~\eqref{ex_uni} for $x^*=(1,-1,\pi)^\top$, $\gamma=1$, $\varepsilon=0.1$, $x(0)=(0,0,0)^\top$.

To illustrate the polynomial decay rate estimate stated in Corollary~\ref{thm_step1rate}.II, consider the function
\begin{equation}\label{P_stab_uni}
P(x)=\|x-x^*\|^4.
\end{equation}
In this case,
$$
\|x_\pi(t)-x^*\|\le \left(\|x^0-x^*\|^{-2}+8\gamma(t-\varepsilon)\right)^{-1/2}\text{ for all }t\ge 0.
$$
Fig.~\ref{fig_uni_poly} illustrates the behavior of trajectories of system~\eqref{ex_uni} for $x^*=(\frac12,-\frac12,\frac{\pi}{2})^\top$,  $\gamma=1$, $\varepsilon=0.1$, $x(0)=(0,0,0)^\top$.

\subsection{Trajectory tracking}\label{sec_exmpls1_track}
For a given curve $x^*(t)\in \mathbb R^3$ on a finite time horizon $t\in [0,T]$,
we will illustrate solutions to the trajectory tracking problem (Problem~2) for system~\eqref{ex_uni} with controls of the form~\eqref{cont_uni} generated by the following potential function:
$$
P(x,t) = \|x-x^*(t)\|^2,\quad x\in{\mathbb R}^3,\; t\in [0,T].
$$

{\em Non-feasible curve.}
Consider the curve $x^*\in C^1([0,20\pi];\mathbb R^3)$: $$x^*(t)= (0.01x^*_{c,1}(0.1t),0.01x^*_{c,2}(0.1t),0)^\top, \; t\in [0,20\pi],$$
where the equations for $x^*_{c,1}(t)$ and $x^*_{c,2}(t)$ are given in~\cite{cat_curve}.
The classical and $\pi_\varepsilon$-solutions of system~\eqref{ex_uni} with the feedback control~\eqref{cont_uni} are shown in Fig.~\ref{fig_uni_cat}. For these simulations, we take
\begin{equation}\label{ini_params}
\varepsilon=0.25,\; \gamma=1, \; x(0)=(-4,0,0)^\top.
\end{equation}
Fig.~\ref{fig_uni_cat} presents considerable oscillations of the $x_1$ and $x_2$ solution components around their reference values $x_1^*(t)$ and $x_1^*(t)$.
Note that in this case the curve $x^*(t)$ is not feasible, i.e. $x=x^*(t)\in \mathbb R^3$, $t\in [0,20\pi]$ is not a solution of system~\eqref{ex_uni} under any choice of admissible controls $u_1$ and $u_2$.
Indeed, the only possibility to satisfy system~\eqref{ex_uni} with $x^*_3(t)\equiv 0$
is to have $x^*_2(t)\equiv {\rm const}$, which does not hold in the considered case.
We will show in the next simulation that the oscillations due to non-feasible character of the reference curve can be significantly reduced if $x^*(t)$ is a solution of the kinematic equations~\eqref{ex_uni}.

{\em Feasible curve.}
Consider now the feasible curve $x^*\in C^1([0,20\pi ];\mathbb R^3)$ such that $$x^*(t)=(x^*_1(t),x^*_2(t),x^*_{3}(t))^\top,\; x^*_i(t) = 0.01x^*_{c,i}(0.01t),\, i=1,2,\; \tan x^*_3(t) = \frac{\dot x^*_2}{\dot x^*_1}.$$
In this case
$x^*_{3}(t)$ satisfies system~\eqref{ex_uni} with $u_1=\pm\sqrt{(\dot x^*_1(t))^2+(\dot x^*_2(t))^2}$ and $u_2=\dot x^*_3(t)$.
To illustrate solutions of the trajectory tracking problem, we apply slightly modified controls of the form
\begin{equation}\label{cont_uni+}
\begin{aligned}
&u_1= u^\varepsilon_1(a(x,t),t)=a_{1}(x,t) +2\sqrt{\frac{\pi|a_{12}(x,t)|}{\varepsilon}}\ {\rm sign}(a_{12}(x,t))\cos{\frac{2\pi t}{\varepsilon}}+\tilde u_1(t),\\
&u_2= u^\varepsilon_2(a(x,t),t)=a_{2}(x,t) +2\sqrt{\frac{\pi|a_{12}(x,t)|}{\varepsilon}}\sin{\frac{2\pi t}{\varepsilon}}+\tilde u_2(t).
\end{aligned}
\end{equation}
Fig.~\ref{fig_uni_cat_feas} shows the behavior of the closed-loop system~\eqref{ex_uni},~\eqref{cont_uni+} with the same initial value and control parameters as in~\eqref{ini_params}.

{\em Unbounded and non-Lipschitz curves.}
Note that the approach of Subsection~\ref{sec_tracking} is also applicable for unbounded
curves which are not continuously differentiable, e.g. $x^*(t)=(t,0.5|t-10|,0)^\top$.
The results of numerical simulations are in Fig.~\ref{fig_uni_abs} with the control parameters~\eqref{ini_params} and $x(0)=(0,0,0)^\top$.
However, the Lipschitz property required in Corollary~\ref{cor_tracking} is important, see Fig.~\ref{fig_uni_quad} with $x^*(t)=(t,0.1t^2,0)^\top$.
As in Fig.~\ref{fig_uni_cat}, some zig-zags are present in Fig.~\ref{fig_uni_abs} due to non-feasible character of the reference curve.

Although our theoretical estimates allow to track even non-feasible curves with any prescribed accuracy, possible practical implementations of this approach should take into account the trade-off between the tracking accuracy and the frequency of switching allowed by the actuators.

\subsection{Obstacle avoidance}\label{sec_exmpls1_obst}
We consider the obstacle avoidance problem (Problem~3) for system~\eqref{ex_uni} in the domain $D\subset \mathbb R^3$ represented as
$$
D=\mathcal W\setminus\bigcup_{j=1}^7\mathcal O_j,\quad
\mathcal W=\{x\in\mathbb R^n:\varphi_0(x)\ge0\},\,\mathcal O_i=\{x\in\mathbb R^n:\varphi_i(x)<0\},
$$
where the cylindric workspace $\mathcal W$ and obstacles $\mathcal O_i$ are defined by the functions  $\varphi_i(x)=(x_1-x_{oi})^2+(x_2-y_{oi})^2-r_i^2$ , $i=\overline{0,7}$,
whose parameters are
$$
\begin{aligned}
&x_{o0}=0,\,y_{o0}=0,\,r_{0}=3.5,\\
&x_{o1}=2,\,y_{o1}=1,\,r_{1}=1,\\
&x_{o2}=0,\,y_{o2}=-0.25,\,r_{2}=0.5,\\
&x_{o3}=-1.5,\,y_{o3}=2,\,r_{3}=0.75,\\
&x_{o4}=-2,\,y_{o4}=0,\,r_{4}=0.75,\\
&x_{o5}=1.5,\,y_{o5}=-2,\,r_{5}=0.75,\\
&x_{o6}=0.5,\,y_{o6}=2.5,\,r_{6}=0.5,\\
&x_{o7}=-1,\,y_{o7}=-2,\,r_{7}=1.
\end{aligned}
$$
The potential function $P(x)$ is constructed in the form~\eqref{P_nav} with the target point $x^*=(-2,1,0)^\top$. In Fig.~\ref{fig_uni_obst} we present the classical and $\pi_\varepsilon$-solutions of the corresponding closed-loop system~\eqref{ex_uni} with $x^0=(1,-1,0)^\top$ and the control~\eqref{cont_uni} with $\varepsilon=0.25$. Fig.~\ref{fig_uni_obst2} shows the closed-loop response with the same initial point and $\varepsilon=0.1$. These figures illustrate that the proposed controllers solve the obstacle avoidance problem with acceptable accuracy.

\newpage
\begin{figure}[h!]
\begin{minipage}{1\linewidth}
\begin{center}
  \includegraphics[width=1\linewidth]{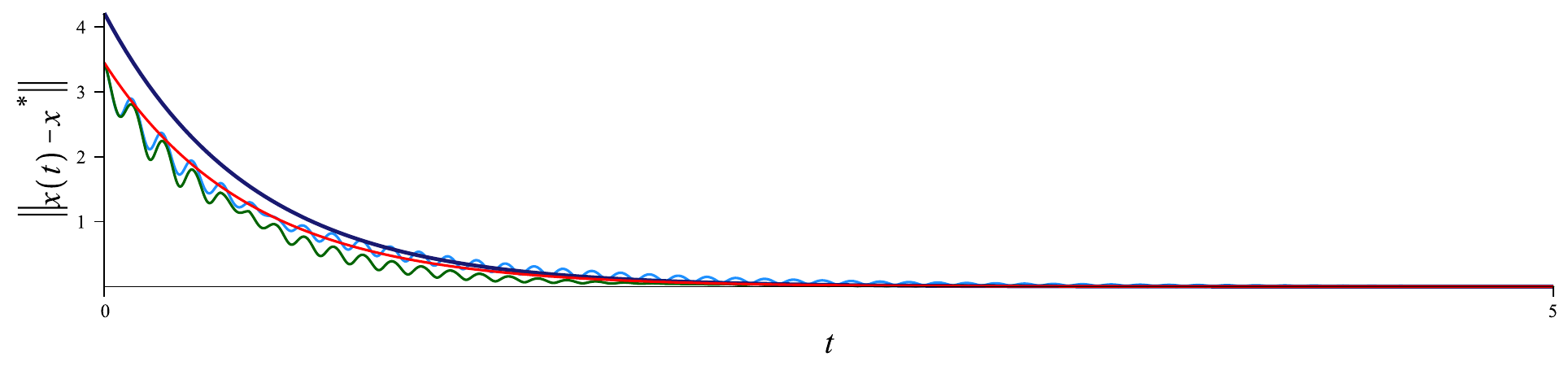}\\
\end{center}
\end{minipage}
\caption{Exponential stabilization: blue -- classical solution, green -- $\pi_\varepsilon$-solution, red -- gradient system $\dot{\bar x}=-\gamma\nabla P(\bar x)$, dark blue -- decay rate $\|x^0-x^*\|e^{-2\gamma(t-\varepsilon)}$}\label{fig_uni_stab}
\begin{minipage}{1\linewidth}
\begin{center}
  \includegraphics[width=1\linewidth]{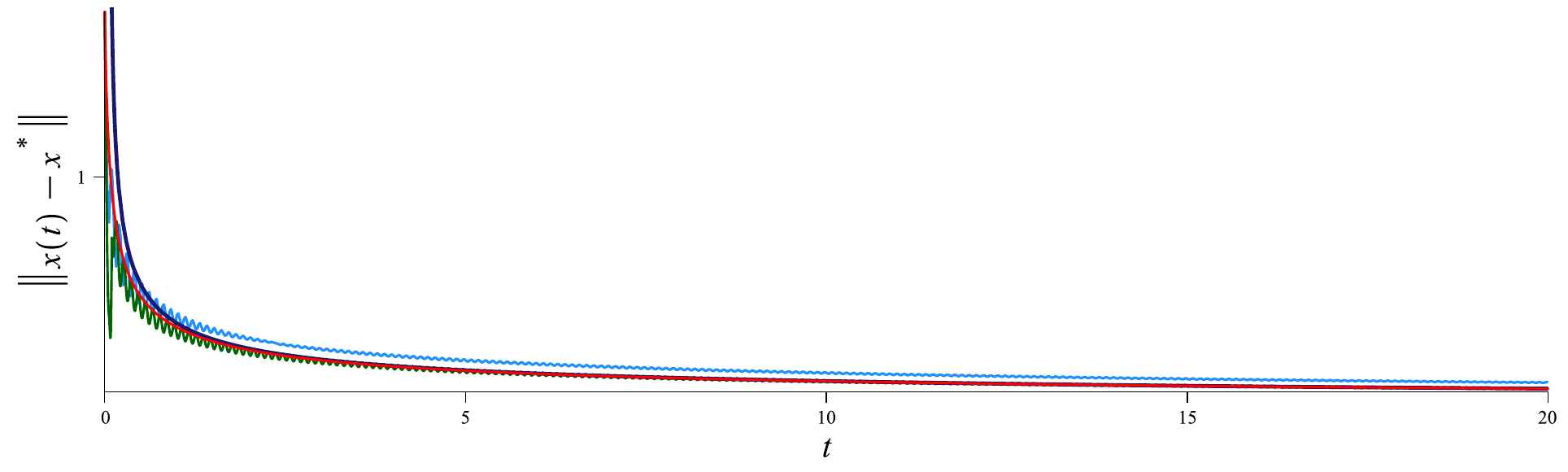}\\
\end{center}
\end{minipage}
\caption{Polynomial stabilization: blue -- classical solution, green -- $\pi_\varepsilon$-solution, red -- gradient system $\dot{\bar x}=-\gamma\nabla P(\bar x)$, dark blue -- decay rate $\left(\|x^0-x^*\|^{-2}+8\gamma(t-\varepsilon)\right)^{-1/2}$}\label{fig_uni_poly}
\begin{minipage}{1\linewidth}
\begin{center}
  \includegraphics[width=1\linewidth]{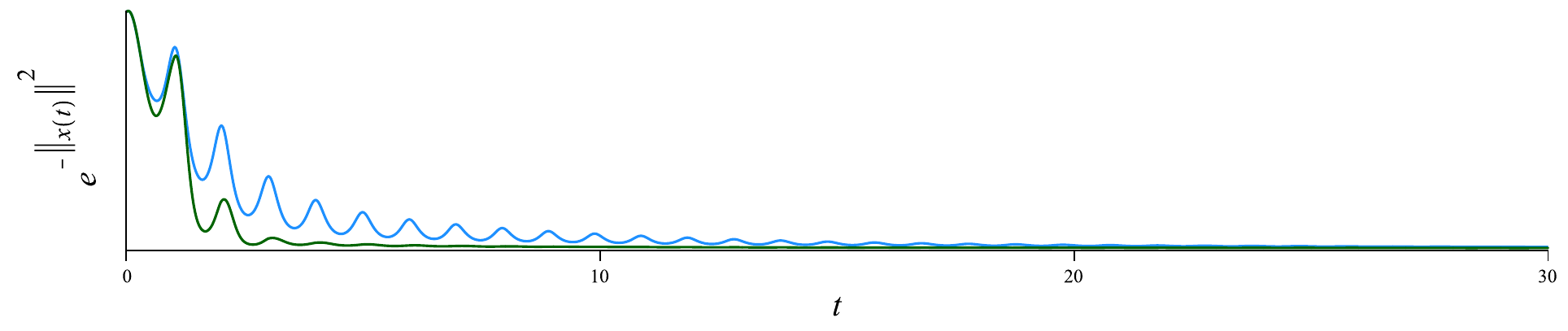}
 \includegraphics[width=1\linewidth]{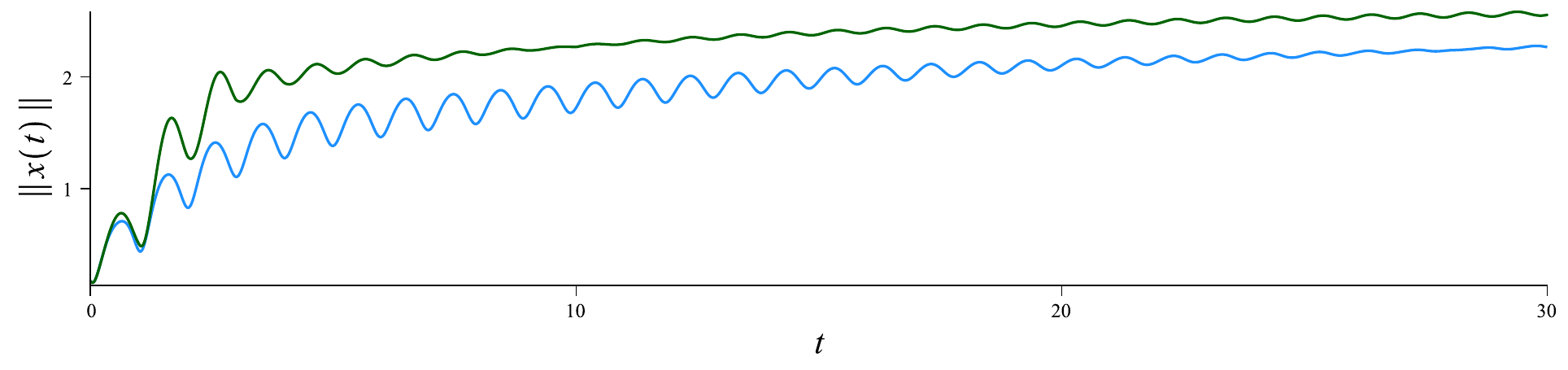}
\end{center}
\end{minipage}
\caption{``'Vanishing velocity property'': blue -- classical solution, green -- $\pi_\varepsilon$-solution}\label{fig_uni_unbound}
\end{figure}

\newpage

\begin{figure}[h!]
\begin{minipage}{0.5\linewidth}
\begin{center}
  \includegraphics[width=1\linewidth]{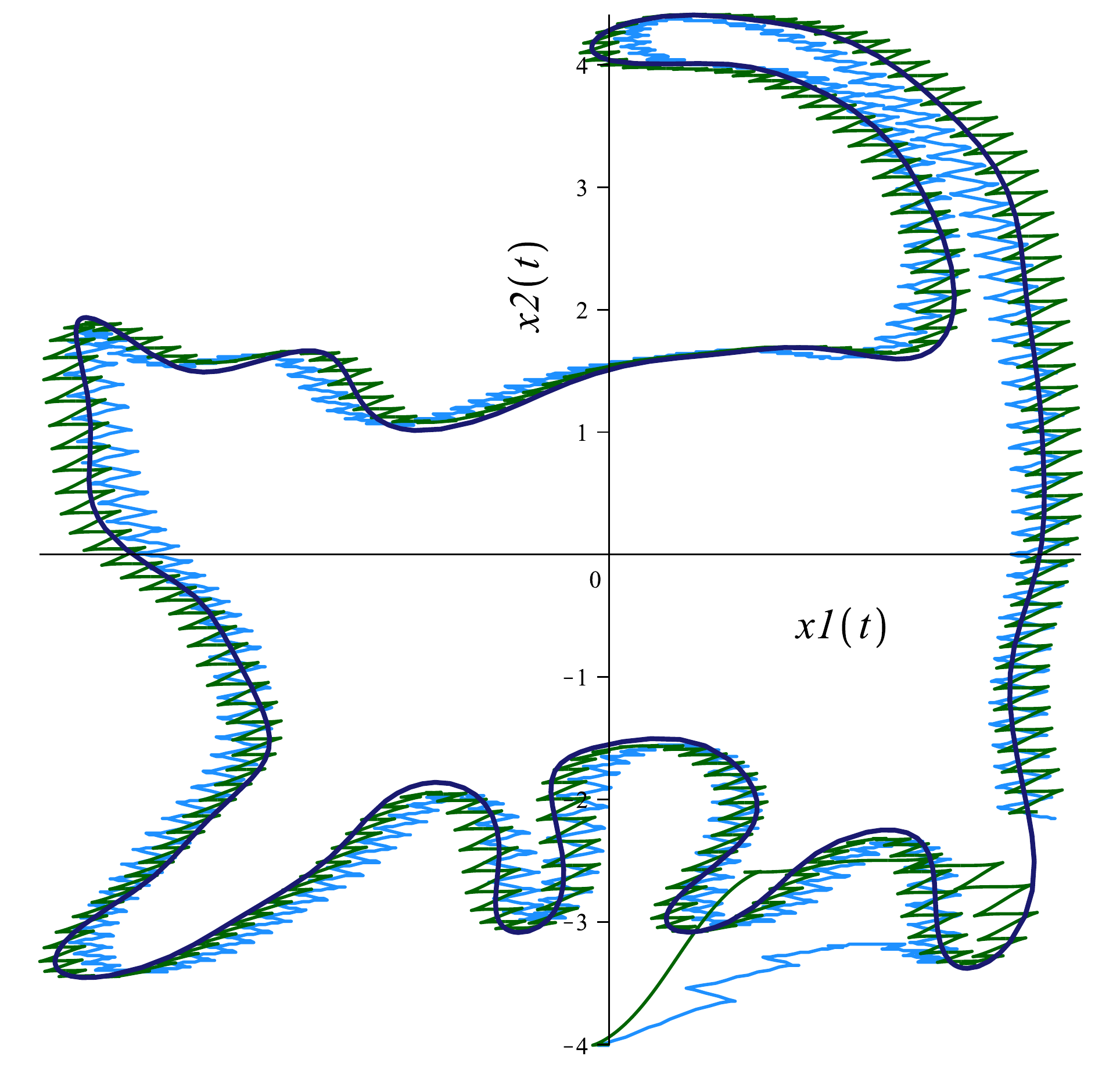}
 \includegraphics[width=1\linewidth]{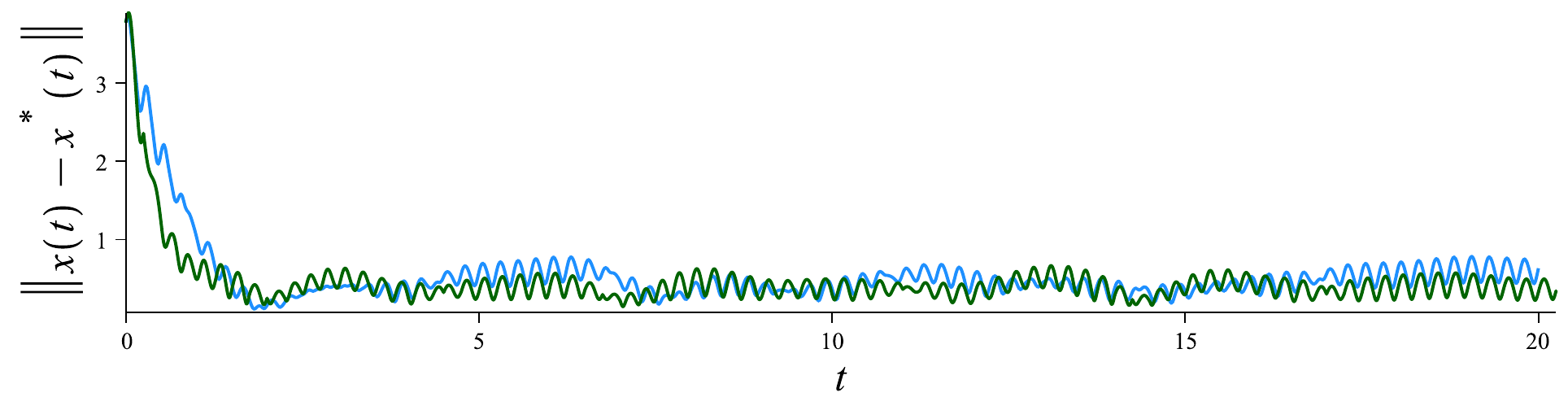}
\end{center}
\caption{blue -- classical solution, green -- $\pi_\varepsilon$-solution, dark blue -- $x^*(t)$, $t=0..20\pi$}\label{fig_uni_cat}
\end{minipage}
\begin{minipage}{0.5\linewidth}
\begin{center}
  \includegraphics[width=1\linewidth]{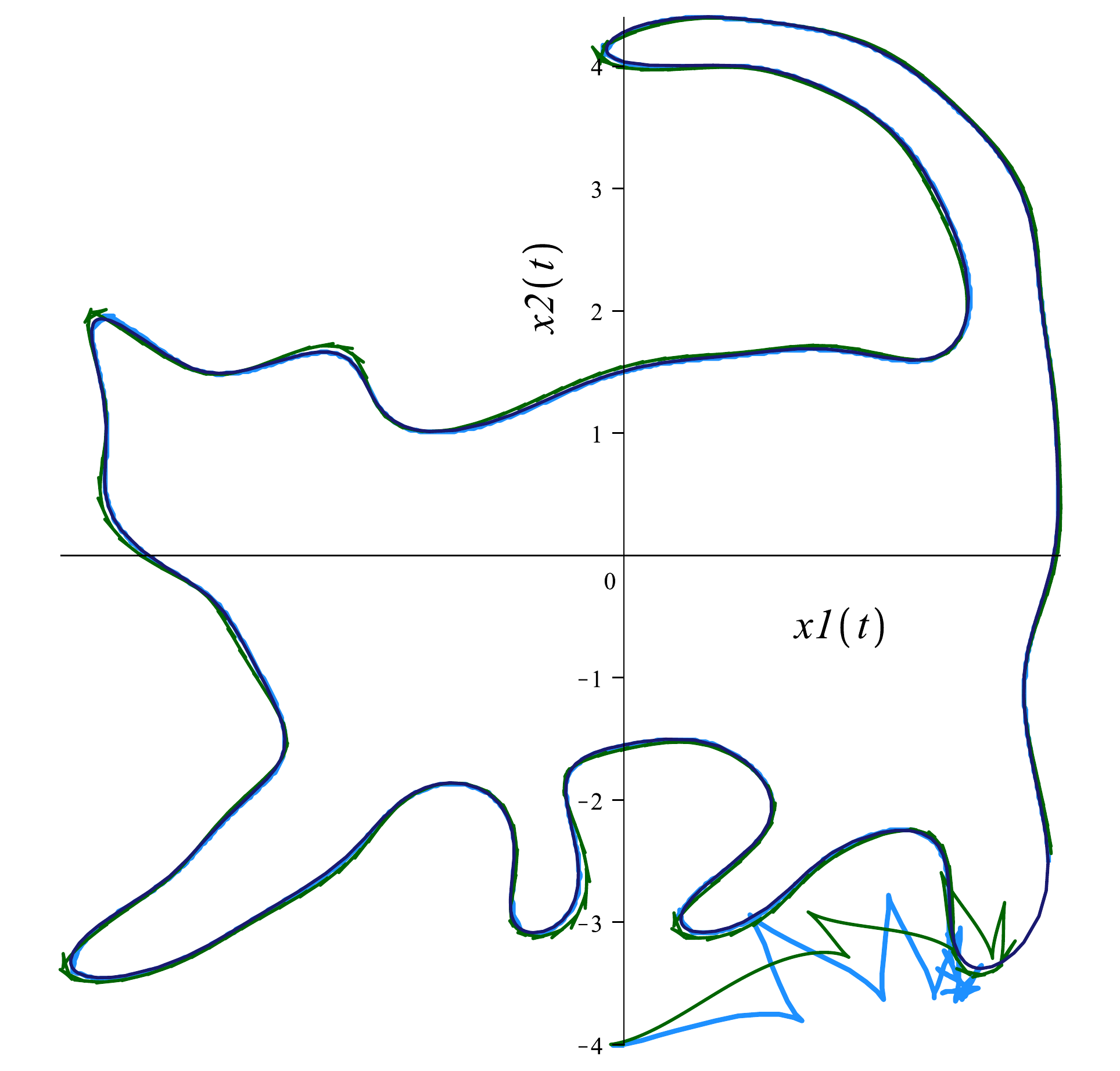}
 \includegraphics[width=1\linewidth]{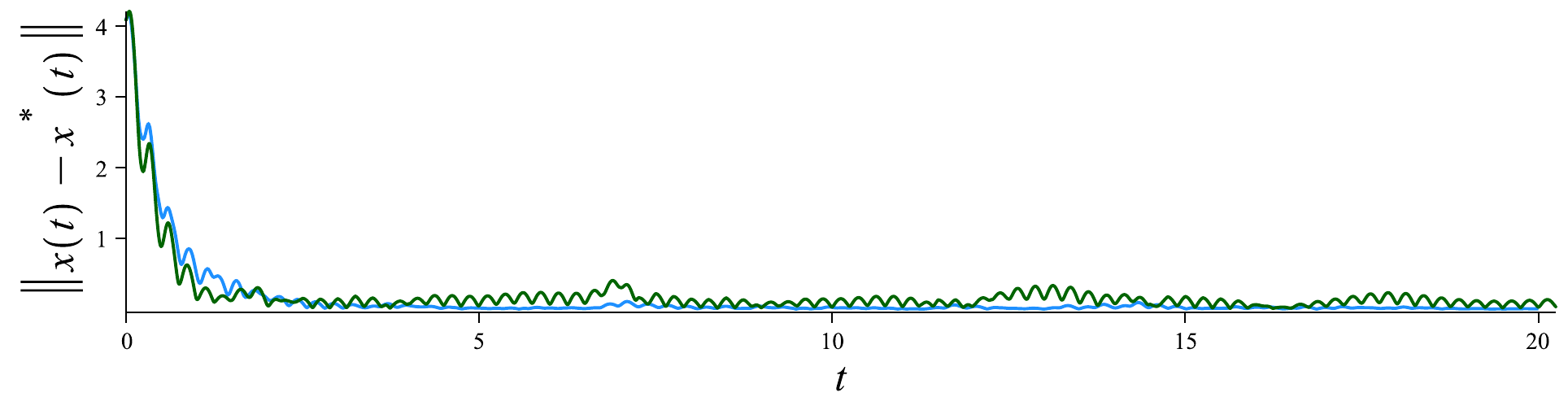}
\end{center}
\caption{blue -- classical solution, green -- $\pi_\varepsilon$-solution, dark blue -- $x^*(t)$, $t=0..20\pi$}\label{fig_uni_cat_feas}
\end{minipage}
\begin{minipage}{0.5\linewidth}
\begin{center}
  \includegraphics[width=1\linewidth]{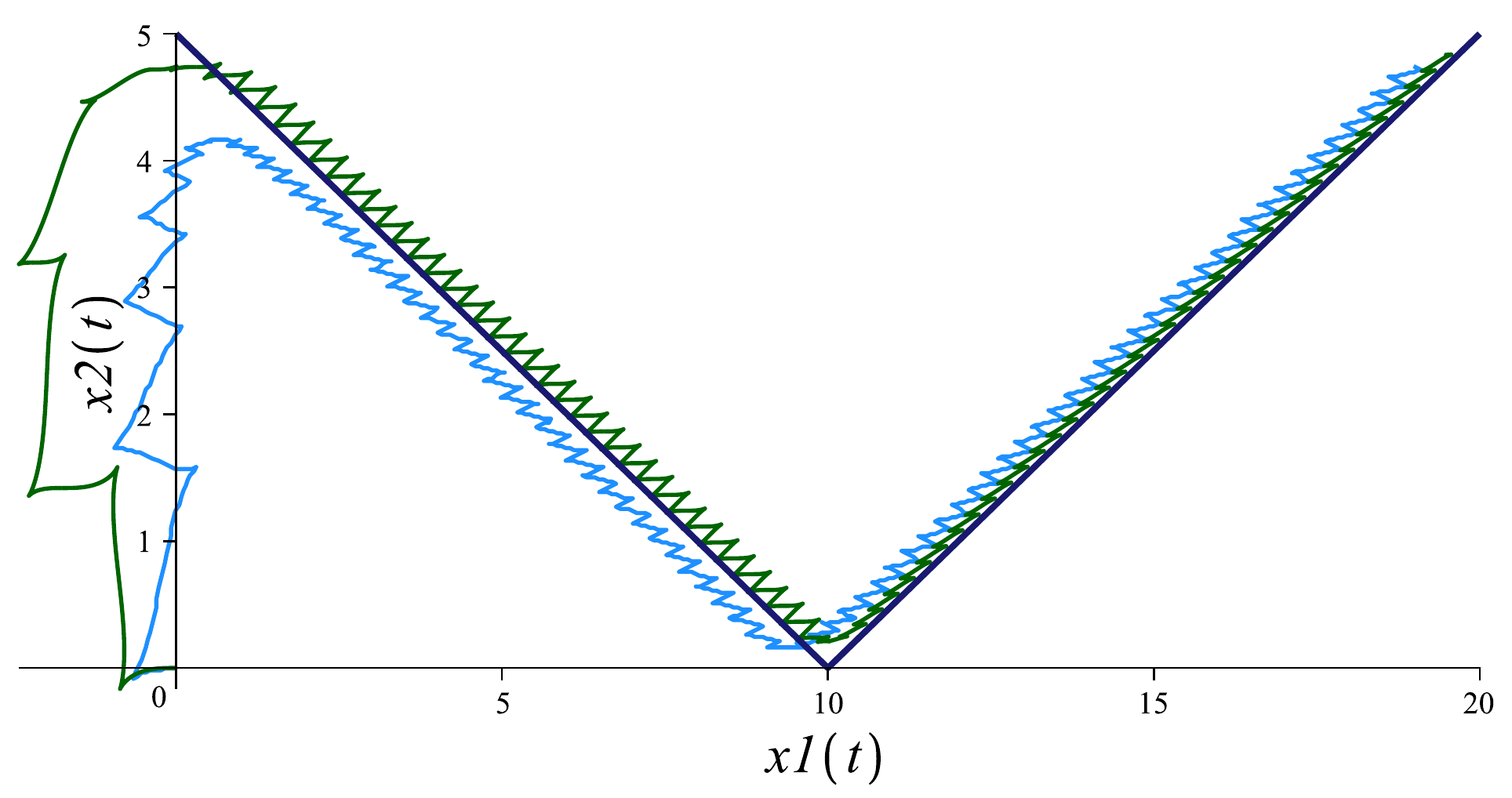}
 \includegraphics[width=1\linewidth]{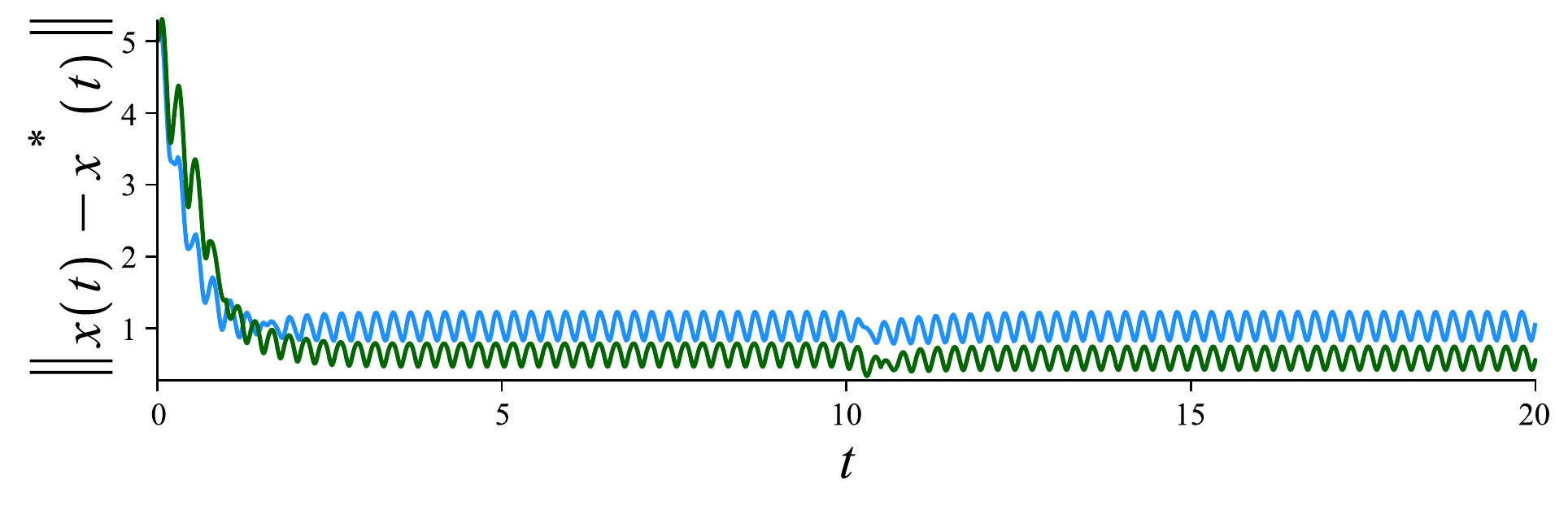}
\end{center}
\caption{blue -- classical solution, green -- $\pi_\varepsilon$-solution, dark blue -- $x^*(t)$}\label{fig_uni_abs}
\end{minipage}
\begin{minipage}{0.5\linewidth}
\begin{center}
  \includegraphics[width=1\linewidth]{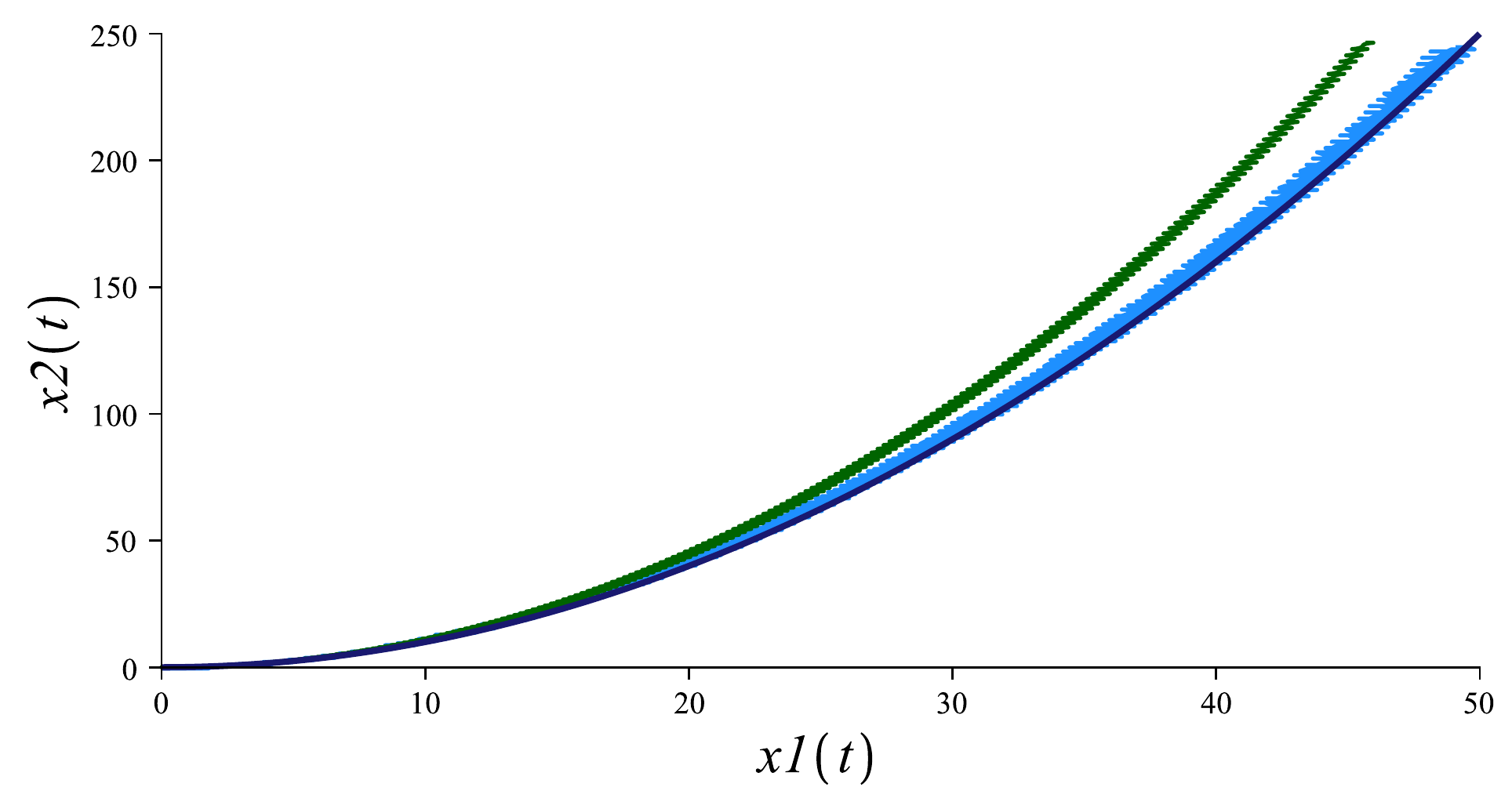}
 \includegraphics[width=1\linewidth]{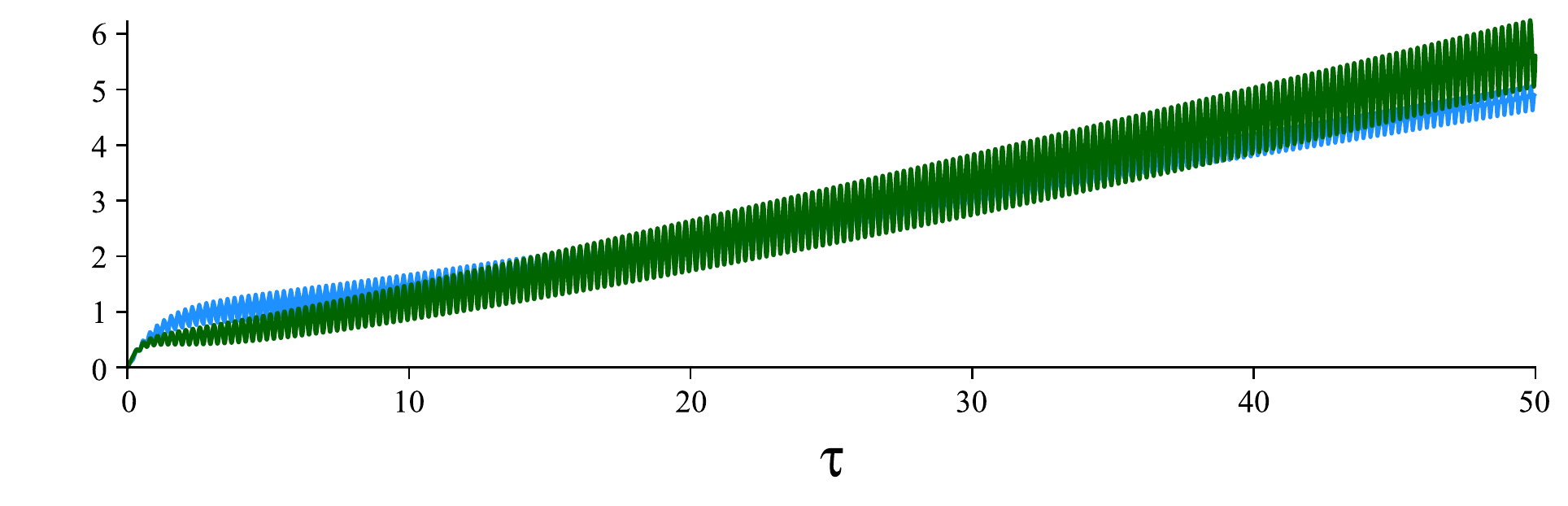}
\end{center}
\caption{blue -- classical solution, green -- $\pi_\varepsilon$-solution, dark blue -- $x^*(t)$}\label{fig_uni_quad}
\end{minipage}
\end{figure}

\newpage
\begin{figure}[ht]
\begin{minipage}{0.5\linewidth}
\begin{center}
  \includegraphics[width=1\linewidth]{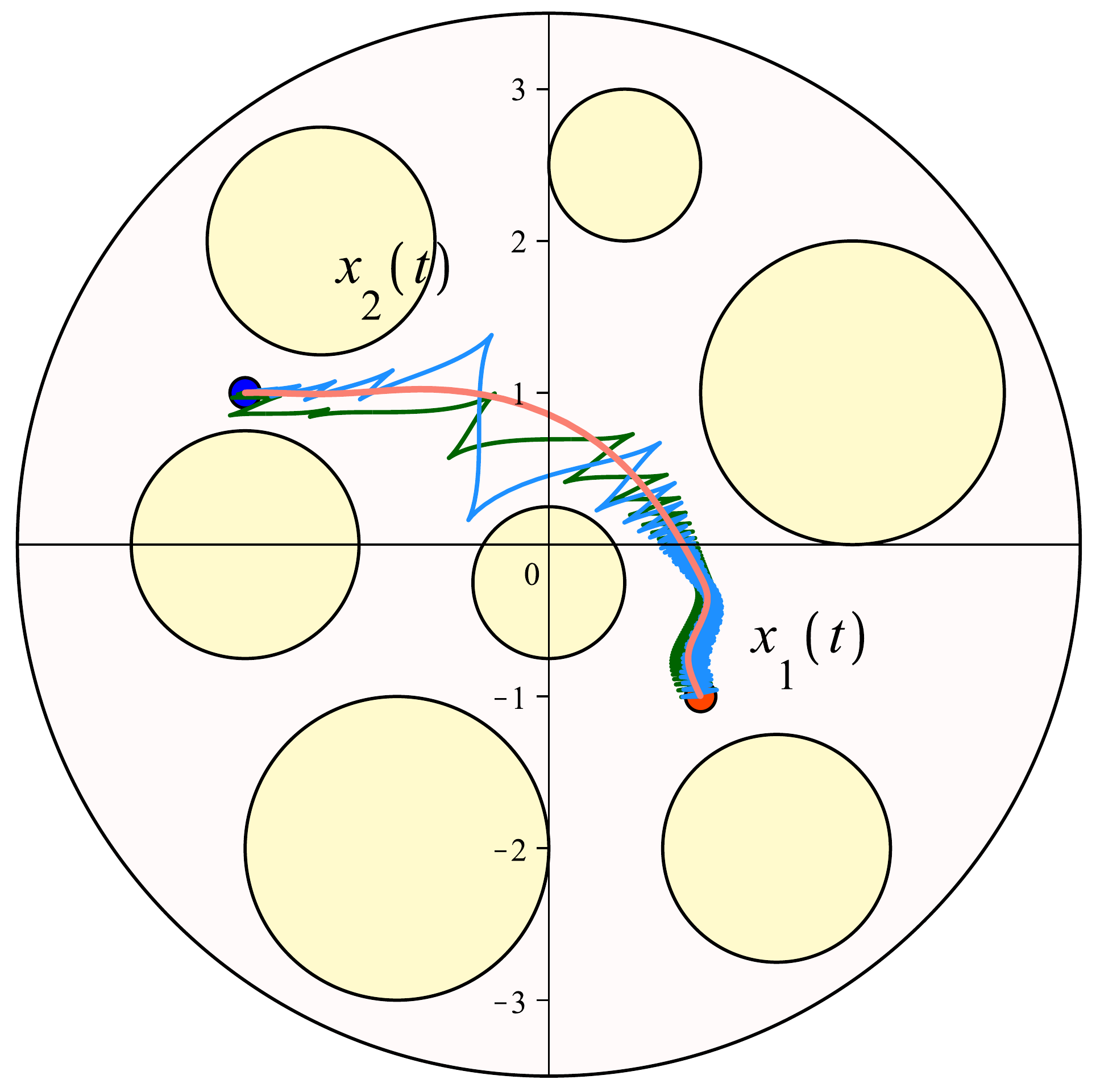}
 \includegraphics[width=1\linewidth]{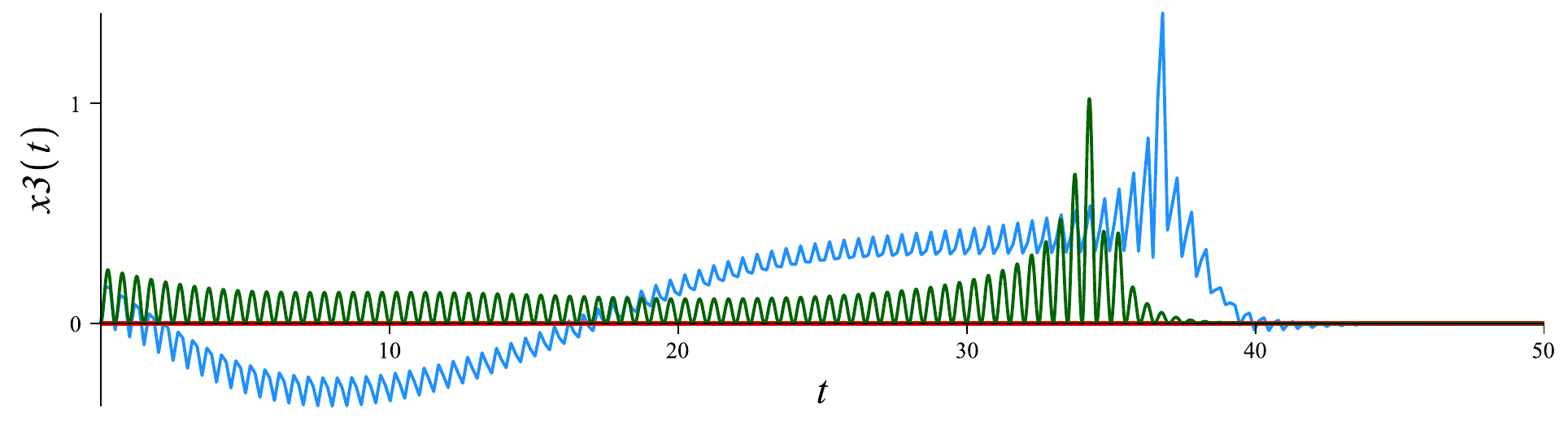}\\
\end{center}
\caption{blue -- classical solution, green -- $\pi_\varepsilon$-solution, red -- gradient system}\label{fig_uni_obst}
\end{minipage}
\begin{minipage}{0.5\linewidth}
\begin{center}
  \includegraphics[width=1\linewidth]{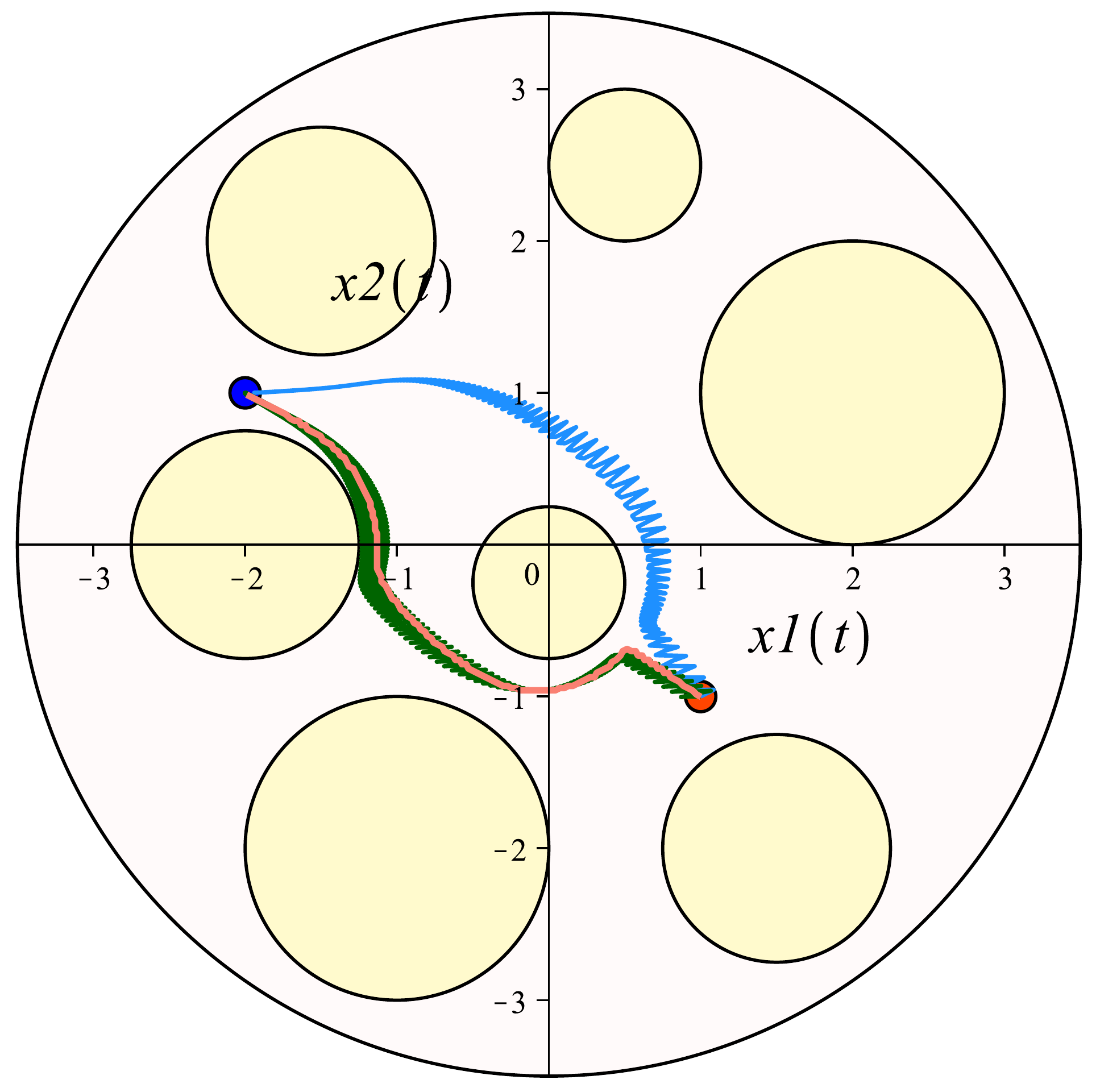}
 \includegraphics[width=1\linewidth]{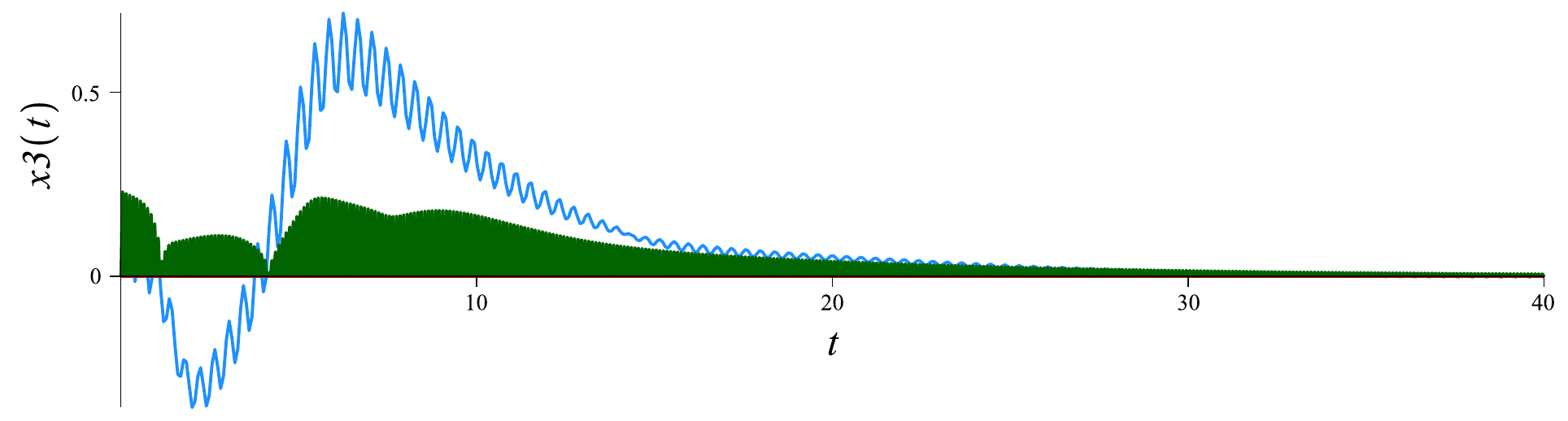}\\
\end{center}
\caption{blue -- classical solution, green -- $\pi_\varepsilon$-solution, red -- gradient system}\label{fig_uni_obst2}
\end{minipage}
\end{figure}

\section{Conclusion}\label{sec_concl}
The proposed design methodology can be considered as a multi-layered hierarchical scheme, where the reference dynamics (upper level) is governed by the gradient flow system~\eqref{sys_grad} with some potential function $P(x,t)$, and the physical level is ruled by nonholonomic control system~\eqref{Sigma} with oscillating inputs~\eqref{cont}.
In this framework, the coordination between the physical and reference dynamics is performed via discrete-time sampling at time instants $t_j = t_0+\varepsilon j$, $j=1,2, ...$~.
The proposed scheme generalizes and significantly extends the approaches previously developed for {\em particular} control problems with {\em time-invariant} vector fields such as stabilization~\cite{ZuSIAM}, motion planning on a finite time horizon~\cite{ZG17}, and obstacle avoidance~\cite{ZG17a}.
It should be emphasized that the contribution of this paper allows the treatment of nonlinear control systems with {\em time-varying} vector fields and {\em relatively simple structure of the control functions}~\eqref{cont}, whose amplitude factors $a(x,t)$ are effectively defined by the matrix inversion in~\eqref{a1}.
The latter feature is considered as an important advantage with respect to the method of~\cite{ZuSIAM,ZG17}, where solutions to a system of nonlinear algebraic equations are required for the design procedure.
Although the formal proof of our results for small $\varepsilon$ is established for $\pi_\varepsilon$-solutions only,
numerical simulations illustrate the similar behavior of classical solutions of the corresponding closed-loop system. Hence, the analysis of asymptotic behavior of classical solutions remains the subject of future study.

\newpage
\appendix
\section{Auxiliary results}\label{sec_aux}
In this appendix, we summarize some auxiliary lemmas which are needed for the proof of the main results.
\begin{lemma}~\label{lemma_x}
 Let $D\subseteq\mathbb R^n$, $t_0\ge 0$, and $x(t)\in D$, $t_0\le t\le\tau$, be a  solution of system~\eqref{Sigma}. Assume that there exist  $M ,L\ge0$ such that
\begin{align*}
\|f_k(x,t)\|\le   M,\,\|f_k(x,t)-f_k(y,t)\|\le  L\|x-y\|,
\end{align*}
for all $x,y\in D$, $t\ge0$, $k=\overline{1,m}$.
    Then
    \begin{equation}\label{est_Z}
      \|x(t)-x(t_0)\|\le (t-t_0) e^{L(t-t_0)}MU,\text{ for all }t\in[t_0,\tau],
    \end{equation}
    with $U=\max\limits_{t\in[t_0,\tau]}\sum\limits_{k=1}^{m}|u_{k}(t)|$.
  \end{lemma}
  \begin{proof}
    Follows from the Gr\"onwall--Bellman inequality.
  \end{proof}
 \begin{lemma}~\label{lemma_volt}
Let $D\subseteq\mathbb R^n$, $t_0\ge 0$, and $x(t)\in D$, $t_0\le t\le\tau$, be a  solution of system~\eqref{Sigma} with $u\in C[t_0,\tau]$ and $x(t_0)=x^0\in D$. Assume that the vector fields $f_k\in C^2(D\times\mathbb R^+;\mathbb R^n)$ are such that
$f_k(\cdot,t)\in C^3(D;\mathbb R^n)$ for each fixed $t\ge 0$,  $k=\overline{1,m}$.
Then $x_\pi(t)$ can be represented in the following way:
$$
\begin{aligned}
  x_\pi(t)=x^0&+\sum_{k=1}^mf_k(x^0,t_0)\int\limits_{t_0}^tu_k(s_1)ds_1 +\sum_{k_1,k_2=1}^mL_{f_{k_2}}f_{k_1}(x^0,t_0)\int\limits_{t_0}^t\int\limits_{t_0}^{s_1}u_{k_1}(s_1)u_{k_2}(s_2)ds_2ds_1\\
  &+r_1(t)+r_2(t),
\end{aligned}
$$
where
\begin{equation}\label{r1}
\begin{aligned}
  r_1(t)=&\sum\limits_{k_1,k_2,k_3=1}^m\int\limits_{t_0}^t\int\limits_{t_0}^{s_1}\int\limits_{t_0}^{s_2} L_{f_{k_3}}L_{f_{k_2}}f_{k_1}(x(s_3),s_3)u_{k_1}(s_1)u_{k_2}(s_2)u_{k_3}(s_3)ds_3ds_2ds_1,\\
  r_2(t)=&\sum_{k=1}^m\int\limits_{t_0}^t\int\limits_{t_0}^{s_1}\frac{\partial}{\partial s_2}f_k(x(s_2),s_2)u_k(s_1)ds_2ds_1\\
  &+\sum_{k_1,k_2=1}^m\int\limits_{t_0}^t\int\limits_{t_0}^{s_1}\int\limits_{t_0}^{s_2}\frac{\partial}{\partial s_3}\big(L_{f_{k_2}}f_{k_1}(x(s_3),s_3)\big)u_{k_1}(s_1)u_{k_2}(s_2)ds_3ds_2ds_1
\end{aligned}
\end{equation}
 \end{lemma}
  \begin{proof}
    This result provides a modification of the Chen--Fliess series expansion (see, e.g.,~\cite{ZG17}).
  \end{proof}

\section{Proof of Lemma~\ref{lem_step1}}\label{proof_step1}

The proof consists of several steps. Throughout the paper, we assume that
 $$
 0<\varepsilon\le \varepsilon_0(\gamma)=\frac{1}{\gamma},
 $$
and $\gamma>0$ will be chosen in Step~3.

\emph{Step~1. The goal of the first step is to find $\varepsilon_1(\gamma)>0$ such that, for all $\gamma>0$ and $\varepsilon\in(0,\min\{\varepsilon_0(\gamma),\varepsilon_1(\gamma)\}]$, the $\pi_\varepsilon$-solution $x_\pi(t)$ of system~\eqref{Sigma} with the initial data $x_\pi(t_0)=x^0$ and the controls $u_k=u_k^\varepsilon(a(x,t),t)$ is well-defined on $t\in[t_0,t_0+\varepsilon]$, i.e. $x_\pi(t)\in D$ for all $t\in[t_0,t_0+\varepsilon]$.}

Let $t_0\ge0$, $x^0\in D$, and let $P(x,t)$ satisfy Assumptions~\ref{as_Pbounds}--\ref{as_Psets}. Given   any positive numbers $\lambda>0$ and $\rho>0$ satisfying Assumption~\ref{as_Psets}, consider the level sets
\begin{equation}\label{tilD}
\widetilde {\mathcal D}_t=\mathcal L_t^{P,P(x^0,t_0)+\lambda}= \{x\in\mathbb R^n:P(x,t)\le P(x^0,t_0)+\lambda\}.
\end{equation}
and
$$
{\mathcal L}_t^{\nabla P,\rho} = \{x\in\mathbb R^n:\|\nabla_xP(x,t)\|\le \rho\}
$$
for $t\ge0$. By Assumption~\ref{as_Psets}, $\widetilde{\mathcal D}_t$ are compact subsets and
 \begin{equation}\label{emb}
     \mathcal L_t^{\nabla P,\rho}\subseteq  \mathcal L_t^{P,P(x^0,t_0)+\lambda}\subset  D\text{ for each }t\ge0.
\end{equation}

Note that according to Definition~\ref{def_pi}, $u_k=u_k^\varepsilon(a(x^0,t_0),t)$ for $t\in[t_0,t_0+\varepsilon]$. Using H\"older's inequality, we estimate the value of $U(x^0,t_0)=\max\limits_{t_0\le t\le t_0+\varepsilon}\sum\limits_{k=1}^m|u_k^\varepsilon(a(x^0,t_0),t|$:
$$
\begin{aligned}
U(x^0,t_0)&\le \max\limits_{t_0\le t\le t_0+\varepsilon}\sum_{i\in S_1}|a_i(x^0,t_0)| +4\sqrt{\frac{\pi}{\varepsilon}}\sum_{(j_1,j_2)\in S_2}\sqrt{|a_{j_1j_2}(x^0,t_0)|\kappa_{j_1j_2}}\\
&\le \|a(x^0,t_0)\||S_1|+\|a(x^0,t_0)\|^{1/2}4\sqrt{\frac{\pi}{\varepsilon}}\Big(\sum_{(j_1,j_2)\in S_2}|\kappa_{j_1j_2}|^{2/3}\Big)^{3/4}.
\end{aligned}
$$
From Assumption~\ref{as_rank}.2) and formula~\eqref{a1} we conclude that, for all $\gamma>0$ and $\varepsilon\in(0,\varepsilon_0(\gamma)]$,
\begin{equation}\label{U1}
 U(x^0,t_0)\le c_u\sqrt{\frac{\gamma{M_F}\|\nabla_xP(x^0,t_0)\|}{\varepsilon}}
\end{equation}
with $c_u=|S_1|\sqrt{{M_F}} L_{Px}+4\sqrt\pi\Big(\sum_{(j_1,j_2)\in S_2}|\kappa_{j_1j_2}|^{2/3}\Big)^{3/4}$, where the constant $L_{Px}$ is defined from Assumption~\ref{as_Pbounds}.2) with $\{\widetilde{\mathcal D}_t\}_{t\ge0}$ given by~\eqref{tilD}. Let us also take constants $M_f>0$ and $L_{fx}\ge0$ from Assumption~\ref{as_f}.1)--\ref{as_f}.2) with $\{\widetilde{\mathcal D}_t\}_{t\ge0}$ given by~\eqref{tilD}:
$$
  \begin{aligned}
    \|f_k(x,t)\|\le M_f,\,\|f_k(x,t)-f_k(y,t)\|\le L_{fx}\|x-y\|\text{ for all }t\ge0,\,x,y\in \widetilde{\mathcal D}_t.
  \end{aligned}
  $$
Then Lemma~\ref{lemma_x} together with~\eqref{U1} yields the following estimate:
\begin{equation}\label{xx0}
\begin{aligned}
  \|x_\pi(t)-x^0\|&\le (t-t_0) e^{L_{fx}(t-t_0)}U(x^0,t_0)M_f\\
  &\le \sqrt{\varepsilon\gamma{M_F}\|\nabla_xP(x^0,t_0)\|} e^{L_{fx}\varepsilon}c_uM_f\text{ for all }t\in[t_0,t_0+\varepsilon].
\end{aligned}
\end{equation}
Let us underline that the latter estimate holds not only for the chosen $x^0\in D$, but also for any $x^0\in \widetilde{\mathcal D}_t$, ${t\ge0}$.
Using the obtained inequality and Assumption~\ref{as_Pbounds}.2), we estimate $P(x_\pi(t),t)$ in the following way:
$$
\begin{aligned}
P(x_\pi(t),t)&\le P(x^0,t_0)+\|P(x_\pi(t),t)-P(x^0,t_0)\|\\
&\le P(x^0,t_0)+L_{Px}\|x_\pi(t)-x^0\|+L_{Pt}|t-t_0|\\
&\le P(x^0,t_0) +\sqrt{\varepsilon\gamma{M_F}} e^{L_{fx}\varepsilon}c_uM_f L_{Px}^{3/2}+\varepsilon L_{Pt},
\end{aligned}
$$
for all $t\in[t_0,t_0+\varepsilon]$.
Let us define $\varepsilon_1(\gamma)$ as the smallest positive root of the equation
\begin{equation}\label{eps1}
\sqrt{\varepsilon\gamma{M_F}} e^{L_{fx}\varepsilon}c_uM_f L_{Px}^{3/2}+\varepsilon L_{Pt}=\lambda.
\end{equation}
Then for any $\gamma>0$ and $\varepsilon\in\big(0,\min\{\varepsilon_0(\gamma),\varepsilon_1(\gamma)\}\big]$,
$$
P(x_\pi(t),t)\le P(x^0,t_0)+\lambda\text{ for all }t\in[t_0,t_0+\varepsilon],
$$
that is $x_\pi(t)\in \widetilde{\mathcal D}_t\subset D$ for all $t\in[t_0,t_0+\varepsilon]$.

\emph{Step 2. The goal of this step is to show that the $\pi_\varepsilon$-solution $x_\pi(t)$ of system~\eqref{Sigma} with the initial data $x_\pi(t_0)=x^0\in D$ and the controls $u_k=u_k^\varepsilon(a(x,t),t)$ can be represented in the form
$$
x_\pi(t_0+\varepsilon)=x^0-\varepsilon\gamma \nabla_x P(x^0,t_0)+R(t_0+\varepsilon),
$$
where $\|R(\varepsilon)\|=O\big((\varepsilon\|\nabla_xP(x^0,t_0)\|)^{1/2}\big)+O((\varepsilon\|\nabla_xP(x^0,t_0)\|)^{3/2})$ as $\varepsilon\to0$.}

Applying Lemma~\ref{lemma_volt} to the $\pi_\varepsilon$-solution $x_\pi(t)$ of system~\eqref{Sigma} with the initial data $x_\pi(t_0)=x^0\in D$ and the controls $u_k=u_k^\varepsilon(a(x,t),t)$ given by~\eqref{cont}, we represent $x_\pi(\varepsilon)$ as
\begin{equation}\label{xpi}
\begin{aligned}
x_\pi(t_0+\varepsilon)=x^0&+\varepsilon\Big(\sum_{i\in S_1}f_i(x^0,t_0)a_i(x^0,t_0)+\sum_{(j_1,j_2)\in S_2}[f_{j_1},f_{j_2}](x^0,t_0)a_{j_1j_2}(x^0,t_0)\Big)\\
&+R(t_0+\varepsilon)=x^0+\varepsilon\mathcal F(x^0,t_0)a(x^0,t_0)+R(t_0+\varepsilon),\\
R(t_0+\varepsilon)=r_1&(t_0+\varepsilon)+r_2(t_0+\varepsilon)+r_3(t_0+\varepsilon),
\end{aligned}
\end{equation}
where $r_1,r_2$ are given by~\eqref{r1} and
$$
\begin{aligned}
r_3(t_0+\varepsilon)=&\varepsilon^{3/2}\sum_{i\in S_1}\sum_{{(j_1,j_2)\in S_2}}[f_i,f_{j_1}](x^0,t_0)a_{i}(x^0,t_0) \sqrt{\frac{|a_{j_1j_2}(x^0,t_0)|}{\pi\kappa_{j_1j_2}}}\\
&+\frac{\varepsilon^2}{2}\sum_{i_1,i_2\in S_1}L_{f_{i_2}}f_{i_1}(x^0,t_0)a_{i1}(x^0,t_0)a_{i2}(x^0,t_0).
\end{aligned}
$$
Using Assumption~\ref{as_f}.2)--\ref{as_f}.3), we  estimate $r_1(\varepsilon),r_2(\varepsilon),r_3(\varepsilon)$ as follows:
$$
\begin{aligned}
\|r_1(t_0+\varepsilon)\|&\le \frac{\varepsilon^3}{6}H_{fx}U^3(x^0,t_0)\le \frac{c_u^3H_{fx}}{6}\big(\varepsilon\gamma{M_F}\|\nabla_xP(x^0,t_0)\|\big)^{3/2},\\
\|r_2(t_0+\varepsilon)\|&\le \frac{\varepsilon^2}{2}L_{ft}U(x^0,t_0)+\frac{\varepsilon^3}{6}H_{ft}U^2(x^0,t_0) \\
&\le \frac{ L_{ft}}{2}\big(\varepsilon^3\gamma{M_F}\|\nabla_xP(x^0,t_0)\|\big)^{1/2}+\frac{ H_{ft}}{6}\big(\varepsilon^2\gamma{M_F}\|\nabla_xP(x^0,t_0)\|\big),\\
\|r_3(t_0+\varepsilon)\|&\le 2\varepsilon^{3/2}L_{2f}\sqrt{|S_1|}\Big(\sum_{(j_1,j_2)\in S_2}\kappa_{j_1j_2}^{-2/3}\Big)^{3/4}\|a(x^0,t_0)\|^{3/2} +\frac{\varepsilon^2}{2}L_{2f}\|a(x^0,t_0)\|^2\\
&\le 2L_{2f}\sqrt{|S_1|}\Big(\sum_{(j_1,j_2)\in S_2}\kappa_{j_1j_2}^{-2/3}\Big)^{3/4}\big(\varepsilon\gamma{M_F}\|\nabla_xP(x^0,t_0)\|\big)^{3/2} \\ &\qquad+\frac{L_{2f}}{2}\big(\varepsilon\gamma{M_F}\|\nabla_xP(x^0,t_0)\|\big)^2.
\end{aligned}
$$
Thus, for any $t_0\ge0$, $x^0\in D$, $\gamma>0$, and $\varepsilon\in\big(0,\min\{\varepsilon_0(\gamma),\varepsilon_1(\gamma)\}\big]$,
\begin{equation}\label{Rest}
\|R(t_0+\varepsilon)\|\le c_{R1}\big(\varepsilon^3\gamma{M_F}\|\nabla_xP(x^0,t_0)\|\big)^{1/2}+c_{R2}\big(\varepsilon\gamma{M_F}\|\nabla_xP(x^0,t_0)\|\big)^{3/2},
\end{equation}
where
$$
\begin{aligned}
c_{R1}=&\frac{L_{ft}}{2}+\frac{H_{ft}}{6}\sqrt{{M_F} L_{Px}},\\
c_{R2}=& \frac{H_{fx}}{6}+2L_{2f}\sqrt{|S_1|}\Big(\sum_{(j_1,j_2)\in S_2}\kappa_{j_1j_2}^{-2/3}\Big)^{3/4}+\frac{L_{2f}}{2}\sqrt{{M_F} L_{Px}}.
\end{aligned}
$$
Finally, inserting~\eqref{a1} into~\eqref{xpi}, we obtain the representation
\begin{equation}\label{xpieps}
  x_\pi(t_0+\varepsilon)=x^0-\varepsilon\gamma \nabla_x P(x^0,t_0)+R(t_0+\varepsilon),
\end{equation}
and thus reach the goal of Step~2.


\emph{Step~3. In this step we will estimate the value $P(x_\pi(t_0+\varepsilon),t_0+\varepsilon)$. Given a $\rho>0$, we will show that there exist an $\varepsilon_2(\gamma)>0$  and a $\bar\gamma(\rho)>0$  such that, for any $\gamma\ge\bar\gamma(\rho)$ and $\varepsilon\in\big(0,\min\{\varepsilon_0(\gamma),\varepsilon_1(\gamma),\varepsilon_2(\gamma)\}\big]$, the $\pi_\varepsilon$-solution $x_\pi(t)$ of system~\eqref{Sigma} with the initial data $x_\pi(t_0)=x^0\in D$ and the controls $u_k=u_k^\varepsilon(a(x,t),t)$ satisfies the property $$P(x_\pi(t_0+\varepsilon),t_0+\varepsilon)\le P(x^0,t_0)\text{ provided that }\|\nabla P_x(x^0,t_0)\|\ge\frac{\rho}{2}.$$}

To analyze the value  $P(x_\pi(t_0+\varepsilon),t_0+\varepsilon)$, we use the Taylor formula with Lagrange's form of the remainder for $P(x_\pi(t_0+\varepsilon),t_0)$:
$$
\begin{aligned}
P(x_\pi(t_0+\varepsilon),t_0+\varepsilon)= &P(x_\pi(t_0+\varepsilon),t_0)+P(x_\pi(t_0+\varepsilon),t_0+\varepsilon)- P_\pi(x(t_0+\varepsilon),t_0)\\
=&P(x^0,t_0)+\nabla_xP(x^0,t_0)(x_\pi(t_0+\varepsilon)-x^0)^\top\\
+\frac{1}{2}\sum_{i,j=1}^n\frac{\partial^2P(x,t_0)}{\partial x_i\partial x_j}&\Big|_{x=x^0+\theta(x_\pi(\varepsilon)-x^0),\theta\in[0,1]} \big({x_\pi}_i(t_0+\varepsilon)-x^0_i\big)\big({x_\pi}_j(t_0+\varepsilon)-x^0_j\big)\\
&+P(x_\pi(t_0+\varepsilon),t_0+\varepsilon)- P(x_\pi(t_0+\varepsilon),t_0).
\end{aligned}
$$
Inserting~\eqref{xpieps} into the obtained representation and using Assumption~\ref{as_Pbounds}.2), \ref{as_Pbounds}.4) with $\{\widetilde{\mathcal D}_t\}_{t\ge0}$ given by~\eqref{tilD}, we obtain
$$
\begin{aligned}
P(x_\pi(t_0+\varepsilon),t_0+\varepsilon)\le &P(x^0,t_0)-\varepsilon\gamma\|\nabla_xP(x^0,t_0)\|^2+\|\nabla_xP(x^0,t_0)\|\|R(t_0+\varepsilon)\|\\
&+\frac{H_{Px}}{2}\big(\varepsilon\gamma\|\nabla_xP(x^0,t_0)\|+\|R(t_0+\varepsilon)\|\big)^2+L_{Pt}\varepsilon.
\end{aligned}
$$
With the use of estimate~\eqref{Rest} we conclude that, for all $\varepsilon\in\big(0,\min\{\varepsilon_0(\gamma),\varepsilon_1(\gamma)\}\big]$,
\begin{equation}\label{Px_epsg}
\begin{aligned}
P(x_\pi(t_0+\varepsilon),t_0+\varepsilon)\le P(x^0,t_0)&-\varepsilon\gamma\|\nabla_xP(x^0,t_0)\|^2\big(1- \sqrt{\varepsilon \gamma} c_{p1}\big)\\
&+\varepsilon^{3/2}\sqrt \gamma\|\nabla_xP(x^0,t_0)\| c_{p2}+L_{Pt}\varepsilon,
\end{aligned}
\end{equation}
where $c_{p1}=c_{R2}\sqrt{{M_F}^3L_{Px}} +H_{Px}\big(1+
c_{R2}^2{M_F}^3L_{Px} +2c_{R1}c_{R2}\varepsilon{M_F}^2\big)$, $c_{p2}=c_{R1}\sqrt{{M_F} L_{Px}}+{H_{Px}}c_{R1}^2{M_F}\varepsilon$.
Therefore,
$$
P(x_\pi(t_0+\varepsilon),t_0+\varepsilon)\le P(x^0,t_0)-\varepsilon\Big(\gamma\|\nabla_xP(x^0,t_0)\|^2\big(1- \sqrt{\varepsilon \gamma} c_{p1}\big)-\|\nabla_xP(x^0,t_0)\| c_{p2}-L_{Pt}\Big).
$$

 Assume that $\|\nabla_x P(x^0,t_0)\|\ge\dfrac{\rho}{2}>0$. Then the above inequality can be rewritten as
$$
P(x_\pi(t_0+\varepsilon),t_0+\varepsilon)\le P(x^0,t_0)-\varepsilon\|\nabla_xP(x^0,t_0)\|^2\Big(\gamma\big(1- \sqrt{\varepsilon \gamma} c_{p1}\big)-\frac{ 2c_{p2}}{\rho}-\frac{4L_{Pt}}{\rho^2}\Big).
$$
Let us fix any $\sigma\in(0,1)$, {$\tilde\gamma>0$}, and put
$$
\varepsilon_2(\gamma)=\frac{(1-\sigma)^2}{\gamma c_{p1}^2},\,\bar\gamma(\rho)=\dfrac{1}{\sigma}\Big(\tilde\gamma+\dfrac{ 2c_{p2}}{\rho}+\dfrac{4L_{Pt}}{\rho^2}\Big).
$$
We obtain that, for any $\gamma\ge\bar\gamma(\rho)$, $\varepsilon\in\big(0,\min\{\varepsilon_0(\gamma),\varepsilon_1(\gamma),
\varepsilon_2(\gamma)\}\big]$,
{\begin{equation}\label{Pdecay}
  P(x_\pi(t_0+\varepsilon),t_0+\varepsilon)\le P(x^0,t_0)-\varepsilon\tilde\gamma\|\nabla_xP(x^0,t_0)\|^2,
\end{equation}
}
that is,
$$
P(x_\pi(t_0+\varepsilon),t_0+\varepsilon)\le P(x^0,t_0)
$$
whenever $\|\nabla_x P(x^0,t_0)\|\ge\dfrac{\rho}{2}$. Moreover, the obtained inequality is strict if $\|\nabla_x P(x^0,t_0)\|>\dfrac{\rho}{2}$.
Similarly to Step~2, we emphasize that the results of the current step hold for any $x^0\in \{\widetilde{\mathcal D}_t\}_{t\ge0}$ provided that the corresponding $\pi_\varepsilon$-solution $x_\pi(t)$ is well-defined in $\{\widetilde{\mathcal D}_t\}_{t\ge0}$ for all $t\in[t_0,t_0+\varepsilon]$.

\emph{Step~4.
The goal of this step is to ensure the following property: after some finite time $t=T\ge0$, the $\pi_\varepsilon$-solution $x_\pi(t)$ of system~\eqref{Sigma} enters the set
$
{\mathcal L}_{t_0+T}^{\nabla P,{\rho/2}}$
and remains in
$
{\mathcal L}_t^{\nabla P,{\rho}}$, ${t\in[t_0+T,t_0+T+\varepsilon]}$.
More precisely, we will show that there exists an $N\in\mathbb N\cup\{0\}$ such that  $\|\nabla_x P(x_\pi(t_0+N\varepsilon),t_0+N\varepsilon)\|\le \dfrac{\rho}{2}$ and, moreover, there exists an $\varepsilon_{3}(\gamma)>0$  such that, for any $\gamma\ge\bar\gamma(\rho)$ and $\varepsilon\in\big(0,\min\{\varepsilon_0(\gamma),\dots,\varepsilon_{3}(\gamma)\}\big]$, the $\pi_\varepsilon$-solution $x_\pi(t)$ of system~\eqref{Sigma} with the initial data $x_\pi(t_0)=x^0$ and the controls $u_k=u_k^\varepsilon(a(x,t),t)$ satisfies the property
$$
\|\nabla_x P(x_\pi(t),t)\|\le{\rho}\text{ for all }t\in[t_0+N\varepsilon,t_0+(N+1)\varepsilon].
$$
}

 We have obtained in Step~3 that $x_\pi(t_0+\varepsilon)\in\mathcal L_{t_0+\varepsilon}^{P,P(x^0,t_0)}$. Applying the results of Step~1 with the same choice of parameters $\varepsilon$ and $\gamma$ and the initial data $x_\pi(t_0+\varepsilon)\in \mathcal L^{P,P(x^0,t_0)}$, ${t\ge0}$, we get $x_\pi(t)\in \widetilde{\mathcal D}_t$ for all $t\in[t_0,t_0+2\varepsilon]$. Furthermore, we may repeat Steps~2--3 and  conclude that
$$P(x_\pi(t_0+2\varepsilon),t_0+2\varepsilon)\le P(x_\pi(t_0+\varepsilon),t_0+\varepsilon)\text{ provided that }\|\nabla P_x(x_\pi(t_0+\varepsilon),t_0+\varepsilon)\|\ge\frac{\rho}{2}.$$

Let us show that there exists an $N\in\mathbb N\cup\{0\}$ such that $\|\nabla_x P(x_\pi(t_0+N\varepsilon),t_0+N\varepsilon)\|\le \dfrac{\rho}{2}$.
  Indeed, assume $\|\nabla_x P(x_\pi(t_0+N\varepsilon),t_0+N\varepsilon)\|> \dfrac{\rho}{2}$ for all  $N\in\mathbb N\cup\{0\}$.
  Then iterating Step~3 and inequality~\eqref{Pdecay}, we conclude that, for any $N\in\mathbb N$,
  $$
 P(x_\pi(t_0+N\varepsilon),t_0+N\varepsilon)\le P(x^0,t_0)-\varepsilon\tilde\gamma\sum_{k=0}^{N-1}\|\nabla_x P(x_\pi(t_0+k\varepsilon),t_0+k\varepsilon)\|^2\le P(x^0,t_0)-\dfrac{N\varepsilon\tilde\gamma\rho^2}{4},
$$
and
$$
 P(x_\pi(t_0+N\varepsilon),t_0+N\varepsilon)-m_P\le P(x^0,t_0)-m_P-\dfrac{N\varepsilon\tilde\gamma\rho^2}{4}.
$$
Obviously, the right-hand side of the latter inequality becomes strictly negative for $N>\Big[\dfrac{4(P(x^0,t_0)-m_P)}{\varepsilon\tilde\gamma\rho^2}\Big]$, while the left-hand side remains non-negative. The obtained contradiction proves, that after the time $T=N\varepsilon$, $N\in\mathbb N\cup\{0\}$, the $\pi_\varepsilon$-solution $x_\pi(t)$ of system~\eqref{Sigma} enters the set
$
{\mathcal L}_{t_0+T}^{\nabla P,{\rho/2}}$

The next goal  is to ensure that the $\pi_\varepsilon$-solution $x_\pi(t)$ of system~\eqref{Sigma}  remains in the family of sets
$
{\mathcal L}_t^{\nabla P,{\rho}}$ for ${t\in[t_0+T,t_0+T+\varepsilon]}$.
Because of Assumption~\ref{as_Psets},
 $x_\pi(t)\in\widetilde{\mathcal D}_t$ for $t\in[t_0+T,t_0+T+\varepsilon]$.
Applying Assumption~\ref{as_Pbounds}.3) with $\{\widetilde{\mathcal D}_t\}_{t\ge0}$ given by~\eqref{tilD}, we get
$$
\begin{aligned}
\|\nabla_x P(x_\pi(t),t)\|\le &\|\nabla_x P(x_\pi(t_0+T),t_0+T)\|\\
&+\|\nabla_x P(x_\pi(t),t)-\nabla_x P(x_\pi(t_0+T),t_0+T)\|\\
\le &\frac{\rho}{2}+L_{2Px}\|x_\pi(t)-x_\pi(t_0+T)\|+L_{2Pt}\|t-T\|.
\end{aligned}
$$
Since the obtained estimate holds for all $t\in[t_0+T,t_0+T+\varepsilon]$, we apply estimate~\eqref{xx0}:
$$
\begin{aligned}
\|\nabla_x P(x_\pi(t),t)\|\le &\frac{\rho}{2}+L_{2Px}\sqrt{\frac{\varepsilon\gamma{M_F} \rho}{2}} e^{L_{fx}\varepsilon}c_uM_f+L_{2Pt}\varepsilon.
\end{aligned}
$$
Let us take $\varepsilon_{3}(\gamma)$ as the smallest positive root of the equation
$$
L_{2Px}\sqrt{\frac{\varepsilon\gamma{M_F} \rho}{2}} e^{L_{fx}\varepsilon}c_uM_f+L_{2Pt}\varepsilon=\frac{\rho}{2}.
$$
Then for any $\gamma\ge\bar\gamma(\rho)$ and $\varepsilon\in\big(0,\min\{\varepsilon_0(\gamma),\dots,\varepsilon_{3}(\gamma)\}\big]$,
$$
\|\nabla_x P(x_\pi(t),t)\|\le \rho\text{ for all }t\in[t_0+T,t_0+T+\varepsilon].
$$

\emph{Step~5. This step summarizes all the obtained results and completes the proof of this lemma.}

From Steps~3 and~4, there exists an $N\in\mathbb N\cup\{0\}$ such that $\|\nabla_x P(x_\pi(t_0+j\varepsilon),t_0+j\varepsilon)\|\ge \dfrac{\rho}{2}$ for $j=0,1,\dots,N-1$, and
$\|\nabla_x P(x_\pi(t_0+T),t_0+T)\|\le \dfrac{\rho}{2}$. Thus,
$$
P(x_\pi(t_0+T),t_0+T)\le P(x_\pi(t_0+(N-1)\varepsilon),t_0+(N-1)\varepsilon)\le ... \le P(x^0,t_0).
$$
and $\|\nabla_x P(x_\pi(t),t)\|\le \rho$ for all $t\in[t_0+T,t_0+T+\varepsilon]$. Consequently, $x_\pi(t)$ is well-defined in $\widetilde{\mathcal D}_t$ for all $t\in[t_0,t_0+T+\varepsilon]$, and
$$
P(x_\pi(t),t)\le \sup\limits_{\xi\in \mathcal L^{\nabla P,\rho}_t}P(\xi,t)\text{ for }t\in[t_0+T,t_0+T+\varepsilon].
$$
Next, consider  two possible scenarios:

S1) $\|\nabla_x P(x_\pi(t_0+T+\varepsilon),t_0+T+\varepsilon)\|\le \dfrac{\rho}{2}$.

Then similarly to Step~4 we have that $\|\nabla_x P(x_\pi(t),t)\|\le \rho$ for  $t\in[t_0+T+\varepsilon,t_0+T+2\varepsilon]$, which implies that $x_\pi(t)$ is well-defined in $\widetilde{\mathcal D}_t$  for all $t\in[t_0,t_0+T+2\varepsilon]$ and
$$
P(x_\pi(t),t)\le  \sup\limits_{\xi\in \mathcal L^{\nabla P,\rho}_t}P(\xi,t)\text{ for }t\in[t_0+T,t_0+T+2\varepsilon].
$$

S2) $\dfrac{\rho}{2}<\|\nabla_x P(x_\pi(t_0+T+\varepsilon),t_0+T+\varepsilon)\|\le \rho$.

Repeating Steps~3--4, we conclude that there exists an integer $N_2\ge N+2$ such that $\|\nabla_x P(x_\pi(t_0+j\varepsilon),t_0+j\varepsilon)\|> \dfrac{\rho}{2}$ for $j=N+1,\dots,N_2-1$, and
$\|\nabla_x P(x_\pi(t_0+N_2\varepsilon),t_0+N_2\varepsilon)\|\le \dfrac{\rho}{2}$. Besides,
$$
\begin{aligned}
P(x_\pi(t_0+N_2\varepsilon),t_0+N_2\varepsilon)&\le P(x_\pi(t_0+(N_2-1)\varepsilon),t_0+(N_2-1)\varepsilon)\le ... \\
&\le P(x_\pi(t_0+T+\varepsilon),t_0+T+\varepsilon) \le  \sup\limits_{\xi\in \mathcal L^{\nabla P,\rho}_{T+\varepsilon}}P(\xi,T+\varepsilon).
\end{aligned}
$$
Obviously,
$$
P(x_\pi(t_0+N_2\varepsilon),t_0+N_2\varepsilon)\le P^*(\rho,\lambda)=\sup_{t\ge t_0+ T} \sup\limits_{\xi\in \mathcal L^{\nabla P,\rho}_{t}}P(\xi,t).
$$
To estimate the values of $P(x_\pi(t),t)$ for $t\in[t_0+T+\varepsilon,t_0+N_2\varepsilon]$,
denote the integer part of $\dfrac{t-t_0}{\varepsilon}$ as $\Big[\dfrac{t-t_0}{\varepsilon}\Big]$ and observe that $0\le t-t_0-\Big[\dfrac{t-t_0}{\varepsilon}\Big]\varepsilon\le\varepsilon$.
Then by Assumption~\ref{as_Pbounds}.1)--\ref{as_Pbounds}.2) and estimate~\eqref{xx0},
$$
\begin{aligned}
P(x_\pi(t),t)\le& P\Big(x_\pi\Big(t_0+\Big[\dfrac{t-t_0}{\varepsilon}\Big]\varepsilon\Big),t_0+\Big[\dfrac{t-t_0}{\varepsilon}\Big]\varepsilon\Big)\\
&+\Big|P(x_\pi(t),t)-P\Big(x_\pi\Big(t_0+\Big[\dfrac{t-t_0}{\varepsilon}\Big]\varepsilon\Big),t_0+\Big[\dfrac{t-t_0}{\varepsilon}\Big]\varepsilon\Big)\Big|\\
\le &P^*(\rho,\lambda)+ L_{Px}\Big\|x_\pi(t)-x_\pi\Big(t_0+\Big[\dfrac{t-t_0}{\varepsilon}\Big]\varepsilon\Big)\Big\|+L_{Pt}\varepsilon\\
\le & P^*(\rho,\lambda)+L_{Px}\sqrt{L_{Px}\varepsilon\gamma{M_F}} e^{L_{fx}\varepsilon}c_uM_f+L_{Pt}\varepsilon.
\end{aligned}
$$%
From~\eqref{eps1}, for any $\gamma\ge\bar\gamma(\rho)$ and $\varepsilon\in\big(0,\min\{\varepsilon_0(\gamma),\varepsilon_1(\gamma),\varepsilon_2(\gamma)\}\big]$,
$$
P(x_\pi(t),t)\le  P^*(\rho,\lambda)+\lambda\text{ for }t\in[t_0+T+\varepsilon,t_0+N_2\varepsilon].
$$

Iterating S1)--S2), we obtain that $x_\pi(t)$ is well-defined in $\widetilde{\mathcal D}_t$ for all $t\ge t_0$ and $x_\pi(t)\in \{x:P(x,t)\le P^*(\rho,\lambda)+\lambda\}$ for $t\ge T+\varepsilon$.
  As $\lambda$ and $\rho$ are assumed arbitrary, the proof of Lemma~1 is completed.
\qed

\section{Proof of Theorem~\ref{thm_step1Pxt}}\label{proof_step1Pxt}

The first two steps and the beginning of the third step of the proof are similar to the proof of Lemma~\ref{lem_step1}.
We summarize the main differences and results  as follows:
\begin{itemize}
  \item For any $\rho>0$ such that $\mathcal L_t^{P,m_P+\rho}\subset D$, $t\ge0$, we define the sets~\eqref{tilD} as
  $$
  \widetilde {\mathcal D}_t=\mathcal L_t^{P,P(x^0,t_0)+m_P+\rho}\subset D.
  $$
  \item $\varepsilon_1(\gamma)$ is the smallest positive root of the equation
  \begin{equation}\label{eps1}
\sqrt{\varepsilon\gamma{M_F}} e^{L_{fx}\varepsilon}c_uM_f L_{Px}^{3/2}+\varepsilon L_{Pt}=\frac{\rho}{4}.
\end{equation}
Similar to the outcome of Step~1, for any  $\gamma>0$  and $\varepsilon\in\big(0,\min\{\varepsilon_0(\gamma),\varepsilon_1(\gamma)\}\big]$, we have
\begin{equation}\label{estPt}
P(x_\pi(t),t)\le P(x^0,t_0)+\frac{\rho}{4}\text{ for all }t\in[t_0,t_0+\varepsilon].
\end{equation}
\item 

For all $\gamma>0$ and $\varepsilon\in(0,\tilde\varepsilon(\gamma)=\min\{\bar\varepsilon_0(\gamma),\varepsilon_1(\gamma)\}]$, 
the $\pi_\varepsilon$-solution $x_\pi(t)$ of system~\eqref{Sigma} with the initial data $x_\pi(t_0)=x^0$ and the controls $u_k=u_k^\varepsilon(a(x,t),t)$ defined by~\eqref{cont} satisfies the property
{\scriptsize
$$
P(x_\pi(t_0+\varepsilon),t_0+\varepsilon)\le P(x^0,t_0)-\varepsilon\Big(\gamma\|\nabla_xP(x^0,t_0)\|^2\big(1- \sqrt{\varepsilon \gamma} c_{p1}\big)-\|\nabla_xP(x^0,t_0)\| c_{p2}-L_{Pt}\Big).
$$}
\end{itemize}
Now we come to the main part of the proof.
Using the above estimate, Assumption~\ref{as_Pbounds}.2) and property~\eqref{Pgrad}, we obtain
\begin{equation}\label{cp3}
\begin{aligned}
P(x_\pi(t_0&+\varepsilon),t_0+\varepsilon)-m_P\\
&\le (P(x^0,t_0)-m_P)\left(1-\varepsilon\gamma\mu(P(x^0,t_0)-m_P)^{\nu-1}\big(1- \sqrt{\varepsilon \gamma} c_{p1}\big)\right)+\varepsilon c_{p3},
\end{aligned}
\end{equation}
where $c_{p3}=L_{Px} c_{p2}+L_{Pt}$, $\nu\ge 1$.
For an arbitrary $\rho>0$, $\gamma^*>0$, let
$$
\bar\gamma(\rho)=\gamma^*+\frac{2^{2\nu} c_{p3}}{\rho^\nu\mu},\, \varepsilon_2(\gamma)=\frac{(\gamma-\bar\gamma(\rho))^2}{\gamma^3 c_{p1}^2}
$$
Then, for any $\gamma>\bar\gamma(\rho)$, $\varepsilon\in\big(0,\min\{\tilde\varepsilon(\gamma),\varepsilon_2(\gamma)\}\big)$, the following properties hold:
\begin{itemize}
\item[i)]
$$
\begin{aligned}
P(x_\pi(t_0&+\varepsilon),t_0+\varepsilon)-m_P\\
&\le (P(x^0,t_0)-m_P)\left(1-\varepsilon\bar\gamma\mu(P(x^0,t_0)-m_P)^{\nu-1}\right)+\varepsilon c_{p3}.
\end{aligned}
$$
  \item[ii)] If $P(x^0,t_0)-m_P\ge\tfrac{\rho}{4}>0$, then
  $$
\begin{aligned}
  P(x_\pi(t_0+\varepsilon),t_0+\varepsilon)-m_P&< (P(x^0,t_0)-m_P)\big(1-\varepsilon\gamma^*\mu(P(x^0,t_0)-m_P)^{\nu-1}\big)\\
  &< P(x^0,t_0)-m_P.
\end{aligned}
  $$
 This means that $x(\varepsilon)\in \mathcal L_t^{P,P(x^0,t_0)}$, ${t\ge 0}$, and by~\eqref{estPt},   $x_\pi(t)\in \widetilde{\mathcal D}_t$ for all $t\in[t_0,t_0+\varepsilon]$.
  \item[iii)] If $P(x^0,t_0)-m_P<\tfrac{\rho}{4} $, then~\eqref{estPt} immediately implies
    $x_\pi(t)\in \widetilde{\mathcal D}_t$ for all $t\in[t_0,t_0+\varepsilon]$, and $P(x_\pi(t_0+\varepsilon),t_0+\varepsilon)-m_P\le\tfrac{\rho}{2}$. Considering again the two cases $ P(x_\pi(t_0+\varepsilon),t_0+\varepsilon)-m_P\ge\tfrac{\rho}{4}$ and $ P(x_\pi(t_0+\varepsilon),t_0+\varepsilon)-m_P<\tfrac{\pi}{4}$, we see that  $x_\pi(t)\in \widetilde{\mathcal D}_t$ for all $t\in[t_0,t_0+2\varepsilon]$.
\end{itemize}
Repeating ii) and iii), we conclude that $x_\pi(t)\in \widetilde{\mathcal D}_t\subset D$ for all $t\ge0$.

It remains to estimate the decay rate of the function $P(x_\pi(t),t)$ as $t\to+\infty$.

I) If $\nu=1$, then, for all $j\in\mathbb N$,
$$
\begin{aligned}
P(x_\pi(t_0+j\varepsilon),t_0+j\varepsilon)-m_P&\le (P(x_\pi(t_0+(j-1)\varepsilon),t_0+(j-1)\varepsilon-m_P)\left(1-\varepsilon\gamma\sigma\mu\right)+\varepsilon c_{p3}\\
&\le (P(x^0,t_0)-m_P)\left(1-\varepsilon\bar\gamma\mu\right)^j+\varepsilon c_{p3}\sum_{i=0}^{j-1}\left(1-\varepsilon\bar\gamma\mu\right)^{i}.
\end{aligned}
$$
Using the property
$$
1-\varepsilon\bar\gamma\mu\le e^{-\varepsilon\bar\gamma\mu}
$$
and calculating
$$
\sum\limits_{i=0}^{j-1}\left(1-\varepsilon\bar\gamma\mu\right)^{i}=\frac{1-\left(1-\varepsilon\bar\gamma\mu\right)^j}{\varepsilon\bar\gamma\mu} <\frac{1}{\varepsilon\bar\gamma\mu},
$$
we obtain
$$
P(x_\pi(t_0+j\varepsilon),t_0+j\varepsilon)-m_P \le  (P(x^0,t_0)-m_P)e^{-j\varepsilon\bar\gamma\mu}+\frac{c_{p3}}{\bar\gamma\mu}.
$$
Under the above choice of $\bar\gamma$,
 for any $\gamma\ge\bar\gamma(\rho)$ and $\varepsilon\in\big(0,\min\{\tilde\varepsilon(\gamma),\varepsilon_2(\gamma)\}\big]$,
\begin{equation}\label{Pest}
P(x_\pi(t_0+j\varepsilon),t_0+j\varepsilon)-m_P\le  (P(x^0,t_0)-m_P)e^{-j\varepsilon\bar\gamma\mu}+\frac{\rho}{2}.
\end{equation}
Hence,
$$
P(x_\pi(t),t)-m_P\le  (P(x^0,t_0)-m_P)e^{-\bar\gamma\mu(t-t_0)}+\frac{\rho}{2}\text{ for each }t=t_0+j\varepsilon,\,j\in\mathbb N\cup\{0\}.
$$
For an arbitrary $t\ge t_0$, estimate~\eqref{estPt} yields
$$
\begin{aligned}
P(x_\pi(t),t)-m_P&\le P\Big(x_\pi\Big(\Big[\frac{t-t_0}{\varepsilon}\Big]\Big),\Big[\frac{t-t_0}{\varepsilon}\Big]\Big)-m_P+\frac{\rho}{2}\\
&\le  (P(x^0,t_0)-m_P)e^{-\bar\gamma\mu\Big[\frac{t-t_0}{\varepsilon}\Big]\varepsilon}+\rho\\
&\le  (P(x^0,t_0)-m_P)e^{-\bar\gamma\mu(t-t_0-\varepsilon)}+\rho\text{ for all }t\ge t_0.
\end{aligned}
$$

II) If $\nu>1$, then, for all $j\in\mathbb N$,
$$
\begin{aligned}
P(x_\pi(t_0&+j\varepsilon),t_0+j\varepsilon)-m_P\le (P(x_\pi(t_0+(j-1)\varepsilon),t_0+(j-1)\varepsilon)-m_P)\\
&\times\left(1-\varepsilon\bar\gamma\mu(P(x_\pi(t_0+(j-1)\varepsilon),t_0+(j-1)\varepsilon)-m_P)^{\nu-1}\right)+\varepsilon c_{p3}.
\end{aligned}
$$
Let us show that there exists an $N\ge 0$ such that
$
P(x_\pi(t_0+j\varepsilon),t_0+j\varepsilon)-m_P\le\dfrac{\rho}{2}.
$

Assume the contrary: $P(x_\pi(t_0+j\varepsilon),t_0+j\varepsilon)-m_P>\frac{\rho}{2}$ for all $j\in\mathbb N\cup \{0\}$.
Then
$$
\begin{aligned}
P(x_\pi&(t_0+j\varepsilon),t_0+j\varepsilon)-m_P\le (P(x_\pi(t_0+(j-1)\varepsilon),t_0+(j-1)\varepsilon)-m_P)\\
&\times\left(1-\varepsilon\mu\Big(\bar\gamma-\frac{2^\nu c_{p3}}{\rho^\nu\mu}\Big)(P(x_\pi(t_0+(j-1)\varepsilon),t_0+(j-1)\varepsilon)-m_P)^{\nu-1}\right)\\
&\qquad\qquad\qquad\qquad\qquad=(P(x_\pi(t_0+(j-1)\varepsilon),t_0+(j-1)\varepsilon)-m_P)\\
&\times\left(1-\varepsilon\mu\gamma^*(P(x_\pi(t_0+(j-1)\varepsilon),t_0+(j-1)\varepsilon)-m_P)^{\nu-1}\right)
\end{aligned}
$$
To obtain decay rate estimates, we exploit the property of a strictly convex function and its tangent line: for any $\theta\in\mathbb R$, $k>0$, $1- \theta\le (1+k\theta)^{-\frac{1}{k}}$. Thus,
$$
\begin{aligned}
1-\varepsilon\mu\gamma^*(&P(x_\pi(t_0+(j-1)\varepsilon),t_0+(j-1)\varepsilon)-m_P)^{\nu-1}\\ &\le\big(1+\varepsilon\mu\gamma^*(\nu-1)(P(x_\pi(t_0+(j-1)\varepsilon),t_0+(j-1)\varepsilon)-m_P)^{\nu-1}\big)^{\frac{1}{1-\nu}}
\end{aligned}
$$
and
$$
\begin{aligned}
P(x_\pi&(t_0+j\varepsilon),t_0+j\varepsilon)-m_P\le (P(x_\pi(t_0+(j-1)\varepsilon),t_0+(j-1)\varepsilon)-m_P)\\
&\le\big((P(x_\pi(t_0+(j-1)\varepsilon),t_0+(j-1)\varepsilon)-m_P)^{1-\nu}+\varepsilon\mu\gamma^*(\nu-1)\big)^{\frac{1}{1-\nu}}\\
&\le\big((P(x^0,t_0)-m_P)^{1-\nu}+j\varepsilon\mu\gamma^*(\nu-1)\big)^{\frac{1}{1-\nu}}\;\;\text{ for all }j\in\mathbb N\cup\{0\}.
\end{aligned}
$$
Then, for $j\ge T=\dfrac{1}{\varepsilon\mu\gamma^*(\nu-1)}\Big(\Big(\dfrac{\rho}{2}\Big)^{1-\nu}-(P(x^0,t_0)-m_P)^{1-\nu}\Big)$, we get
$$
P(x_\pi(t_0+j\varepsilon),t_0+j\varepsilon)-m_P\le\frac{\rho}{2},
$$
which gives  contradiction. Thus, there exists an $N\in\mathbb N\cup\{0\}$ such that
$$
\begin{aligned}
P(x_\pi&(t_0+j\varepsilon),t_0+j\varepsilon)-m_P\\
&\le\big((P(x^0,t_0)-m_P)^{1-\nu}+j\varepsilon\mu\gamma^*(\nu-1)\big)^{\frac{1}{1-\nu}}\;\;\text{ for all }j=0,1,\dots,N,
\end{aligned}
$$
and
$$
P(x_\pi(t_0+N\varepsilon),t_0+N\varepsilon)-m_P\le\frac{\rho}{2}.
$$

For an arbitrary $t\in[t_0,t_0+N\varepsilon]$, we again exploit  the property
$$
\begin{aligned}
P(x_\pi(t),t)&\le P\Big(x_\pi\Big(\Big[\frac{t-t_0}{\varepsilon}\Big]\Big),\Big[\frac{t-t_0}{\varepsilon}\Big]\Big)+\frac{\rho}{2}\\
&\le \big((P(x^0,t_0)-m_P)^{1-\nu}+\Big[\frac{t-t_0}{\varepsilon}\Big]\varepsilon\mu\gamma^*(\nu-1)\big)^{\frac{1}{1-\nu}}+\frac{\rho}{2}\\
&\le \big((P(x^0,t_0)-m_P)^{1-\nu}+\mu\gamma^*(\nu-1)(t-t_0-\varepsilon)\big)^{\frac{1}{1-\nu}}+\frac{\rho}{2}.
%
\end{aligned}
$$
Similarly to the derivation of~\eqref{estPt}, we can show that, for any $\varepsilon\in\big(0,\min\{\varepsilon_0(\gamma),\varepsilon_1(\gamma)\}\big]$,
$$
P(x_\pi(t),t)-m_P\le P(x_\pi(t_0+N\varepsilon),t_0+N\varepsilon)-m_P+\frac{\rho}{2}\le \rho\text{ for all }t\in[t_0+N\varepsilon,t_0+(N+1)\varepsilon].
$$
Then two cases are possible:

\begin{itemize}
  \item if $P(x_\pi(t_0+(N+1)\varepsilon),t_0+(N+1)\varepsilon)-m_P\le \dfrac{\rho}{2}$, then $P(x_\pi(t),t)-m_P\le \rho$ for all $t\in[t_0+(N+1)\varepsilon,t_0+(N+2)\varepsilon]$;
  \item if $\dfrac{\rho}{2}<P(x_\pi(t_0+(N+1)\varepsilon),t_0+(N+1)\varepsilon)-m_P\le \rho$, then
   $$
\begin{aligned}
P&(x_\pi(t),t)-m_P \\
&\le \big((P(x_\pi(t_0+(N+1)\varepsilon),t_0+(N+1)\varepsilon)-m_P)^{1-\nu}+\mu\gamma^*(\nu-1)(t-t_0-\varepsilon)\big)^{\frac{1}{1-\nu}}\\
&\le \big(\rho^{1-\nu}+\mu\gamma^*(\nu-1)(t-t_0-\varepsilon)\big)^{\frac{1}{1-\nu}}\le\rho,\;\text{for all}\;t\in[t_0+(N+1)\varepsilon,t_0+(N+2)\varepsilon].
\end{aligned}
$$
\end{itemize}
The iteration of the above two cases yields
$$
P(x_\pi(t),t)-m_P\le \rho\text{ for all }t\ge t_0+N\varepsilon,
$$
which completes the proof of Theorem~\ref{thm_step1Pxt}.
\qed

\section{Proof of Theorem~\ref{thm_step1stab}} \label{proof_step1stab}

For an arbitrary $x^0\in D$, define $\mathcal L^{P,P(x^0)}=\{x\in\mathbb R^n:P(x)\le P(x^0)\}$. From the condition~\ref{thm_step1stab}.1),
$$
\mathcal L^{P,P(x^0)} \subseteq \{x\in\mathbb R^n:\|x-x^*\|\le w_{11}^{-1}(P(x^0))\}\subset D.
$$
Let $\widetilde D$ be an arbitrary convex compact set such that
$$
\{x\in\mathbb R^n:\|x-x^*\|\le w_{11}^{-1}(P(x^0))\}\subset \widetilde D\subseteq D.
$$
All the  assumptions of Theorem~\ref{thm_step1Px} are satisfied, so that we immediately have the following properties: for any $\gamma>0$ there exists an $\bar\varepsilon:\mathbb R_{>0}\to\mathbb R_{>0}$ such that, for all $t_0\ge 0$,  $x^0\in \widetilde D$, $t_0\ge0$ and $\varepsilon\in(0,\tilde\varepsilon(\gamma)]$, the $\pi_\varepsilon$-solution $x_\pi(t)$ of system~\eqref{Sigma_x} with the initial data $x_\pi(t_0)=x^0$ and the controls $u_k=u_k^\varepsilon(a(x,t),t)$ given by~\eqref{cont}--\eqref{dithers} are well-defined in $\widetilde D$ for all $t\ge t_0$, and
$$
\lim\limits_{t\to\infty}(P(x_\pi(t))-m_P)=0.
$$
As $\|x_\pi(t)-x^*\|\le w_{11}^{-1}\big(P(x_\pi(t))-m_P\big)$, this also implies
$$
\lim\limits_{t\to\infty}\|x_\pi(t)-x^*\|=0.
$$
Thus, the point $x^*$ is attractive for system~\eqref{Sigma_x}.

Let us prove that $x^*$ is stable. Assume that $\gamma$ and $\bar\varepsilon(\gamma)$ are fixed, $\gamma\bar\varepsilon(\gamma)\le 1$.
For an arbitrary $t\ge t_0$, denote the integer part of $\frac{t-t_0}{\varepsilon}$ as $N=\Big[\frac{t-t_0}{\varepsilon}\Big]$.
From~\eqref{xx0},
$$
\begin{aligned}
  \|x_\pi(t)-x^0\|&\le \sqrt{\varepsilon\gamma{M_F}\|\nabla P(x_\pi(t_0+N\varepsilon))\|} e^{L_{fx}\varepsilon}c_uM_f\\
  &\le \sqrt{{M_F}\|\nabla P(x_\pi(t_0+N\varepsilon))\|} e^{L_{fx}\bar\varepsilon}c_uM_f.
\end{aligned}
$$
Using the triangle inequality and condition~\ref{thm_step1stab}.2), we get
\begin{equation}\label{star}
  \|x_\pi(t)-x^*\|\le \|x_\pi(t_0+N\varepsilon)-x^*\|+ \sqrt{{M_F}w_2\big(\|x_\pi(t_0+N\varepsilon)-x^*\|\big)} e^{L_{fx}\bar\varepsilon}c_uM_f.
\end{equation}
Furthermore, from the proofs of Lemma~\ref{lem_step1} and Theorem~\ref{thm_step1Pxt} it follows that
$$
P(x_\pi(t_0+N\varepsilon)\le P(x^0),
$$
i.e.
\begin{equation}\label{2star}
\|x_\pi(t_0+N\varepsilon)-x^*\|\le w_{11}^{-1}\big(P(x^0)-m_P\big)\le w_{11}^{-1}\big(w_{12}(\|x^0-x^*\|)\big).
\end{equation}
Combining~\eqref{star} and~\eqref{2star} we conclude that, given an arbitrary $\epsilon>0$, one can choose a $\delta>0$ satisfying
$$
 w_{11}^{-1}\big(w_{12}(\delta)\big)+ \sqrt{{M_F}w_2\big(w_{11}^{-1}\big(w_{12}(\delta)\big)\big)} e^{L_{fx}\bar\varepsilon}c_uM_f\le \epsilon,
$$
so that
$$
 \|x_\pi(t)-x^*\|\le \epsilon\text{ for any }t\ge t_0, x^0\in \overline{B_\delta(x^*)}.
$$
\qed

\bibliographystyle{spmpsci}      

\bibliography{biblio_nonh}

\section*{Declarations}
\subsection*{Funding}
This work was partially supported by the DFG (German Research Foundation) under grants GR~5293/1-1 and ZU~359/2-1.
\subsection*{Competing interests}
The authors have no relevant financial or non-financial interests to disclose.
\subsection*{Author contributions}
The authors contributed equally to this work. The authors have read and approved the final manuscript.
\subsection*{Data availability}
The datasets generated during the current study are available from the corresponding author on reasonable request.

\end{document}